\documentclass[12pt]{article}
\usepackage{amsmath}
\usepackage{amsthm,amsfonts,amssymb,epsfig,graphics}
\usepackage{pdfsync}
\usepackage{wrapfig}
\usepackage[utf8]{inputenc}
\usepackage[babel=true,kerning=true]{microtype} 
\usepackage{pgf,tikz}
\usetikzlibrary{arrows}

\usepackage[colorlinks=true,urlcolor=blue]{hyperref}
\usepackage{multicol}

\def\cal{\mathcal}

\newcommand{\bbD}{\mathbb{D}}
\newcommand{\bbH}{\mathbb{H}}
\newcommand{\bbN}{\mathbb{N}}
\newcommand{\bbR}{\mathbb{R}}
\newcommand{\bbS}{\mathbb{S}}

\newcommand{\bbZ}{\mathbb{Z}}

\def\cA{{\cal A}}  \def\cG{{\cal G}}  \def\cS{{\cal S}}
  \def\cH{{\cal H}}  \def\cT{{\cal T}}
\def\cC{{\cal C}}   \def\cO{{\cal O}} 
\def\cD{{\cal D}}    
    
\def\cF{{\cal F}}

\renewcommand{\phi}{\varphi}
\renewcommand{\epsilon}{\varepsilon}

\def\inte{\mathrm{Int}}

\def\homeo{\mathrm{Homeo}}

\newtheorem*{ques*}{Question}
\newtheorem*{prop*}{Proposition}
\newtheorem*{conj*}{Conjecture}
\newtheorem*{theo*}{Theorem}
\newtheorem{coro}{Corollary}[section]

\newtheorem{prop}[coro]{Proposition}
\newtheorem{lemm}[coro]{Lemma}
\newtheorem*{lemm*}{Lemma}
\newtheorem*{coro*}{Corollary}
\def\?{\footnote{?}}

\newcommand{\footnotefred}[1]{}

\newcounter{prblm}
\renewcommand{\theprblm}{\arabic{prblm}}

\newlength{\espaceavantenonce}
\newlength{\espaceapresenonce}
{\vskip \espaceapresenonce}

\newenvironment{enonce3*}[1]{
\vskip\espaceavantenonce \noindent \textbf{\textit{#1.---}} }%
{\vskip \espaceapresenonce}

\newcounter{numexo}
\newcounter{numsubexo}
\newcounter{numsubsubexo}
\newcommand{\subexo}{\stepcounter{numsubexo}
	\setcounter{numsubsubexo}{0}
	\noindent\textbf{\arabic{numsubexo}. }}

\newenvironment{exercise}{\bigskip \hrule \smallskip
\setcounter{numsubexo}{0}	\setcounter{numsubsubexo}{0}
\footnotesize
\noindent\stepcounter{numexo}
\textbf{Exercise \arabic{numexo}.---}}{\normalsize \smallskip \hrule \bigskip}

\begin{document}

\author{Fr\'ed\'eric Le Roux}
 \title{An introduction to Handel's homotopy Brouwer theory}
\maketitle


\begin{abstract}
Homotopy Brouwer theory is a tool to study the dynamics of surface homeomorphisms.
We introduce and illustrate the main objects of homotopy Brouwer theory, and provide a proof of Handel's fixed point theorem. These are the notes of a mini-course held during the workshop ``Superficies en Montevideo'' in March 2012.
\end{abstract}

\tableofcontents

\newpage
\section*{Introduction}
\addcontentsline{toc}{section}{Introduction}

These notes may be seen as a walk around Handel's proof of the following theorem.

\begin{theo*}[Handel's fixed point theorem,\cite{handel1999fixed}]
Consider a homeomorphism $f:\bbD^2\to \bbD^2$ of the closed $2$-disk. Assume the following hypotheses.
\begin{enumerate}
\item[$(H_{1})$] There exists $r\geq 3$ points $x_{1}, \dots ,x_{r} $ in the interior of $\bbD^2$ and $2r$ pairwise distinct points $\alpha_{1}, \omega_{1}, \dots , \alpha_{r}, \omega_{r}$ on the boundary $\partial \bbD^2$ such that, for every $i=1, \dots, r$,
$$
\lim_{n \to -\infty} f^n(x_{i}) = \alpha_{i}, \ \ \  \lim_{n \to +\infty} f^n(x_{i}) = \omega_{i}.
$$
\item[$(H_{2})$] The cyclic order on $\partial \bbD^2$ is as represented on the picture below :
$$
\alpha_{1}, \omega_{r} , \alpha_{2} , \omega_{1}, \alpha_{3}, \omega_{2}, \dots , \alpha_{r}, \omega_{r-1}, \alpha_{1}.
$$
\end{enumerate}
Then $f$ has a fixed point in the interior of $\bbD^2$.
\end{theo*}

\smallskip 

\begin{center}
\fbox{
\begin{tikzpicture}[scale=1]
\tikzstyle{fleche}=[>=latex,->, thick, dotted]
\draw (0,0) circle (2);
\foreach \r in {3}
{
\foreach \i in {1,...,\r}
{\draw [fleche] (\i*360/\r:2) -- (\i*360/\r+480/\r:2) ;
\draw (\i*360/\r:2.4) node {$\alpha_{\i}$};
\draw  (\i*360/\r+480/\r:2.4) node {$\omega_{\i}$};}
}
\end{tikzpicture} 
\hspace{1cm}
\begin{tikzpicture}[scale=1]
\tikzstyle{fleche}=[>=latex,->, thick, dotted]
\draw (0,0) circle (2);
\foreach \r in {4}
{
\foreach \i in {1,...,\r}
{\draw [fleche] (\i*360/\r:2) -- (\i*360/\r+480/\r:2) ;
\draw (\i*360/\r:2.4) node {$\alpha_{\i}$};
\draw  (\i*360/\r+480/\r:2.4) node {$\omega_{\i}$};}
}
\end{tikzpicture} 
}

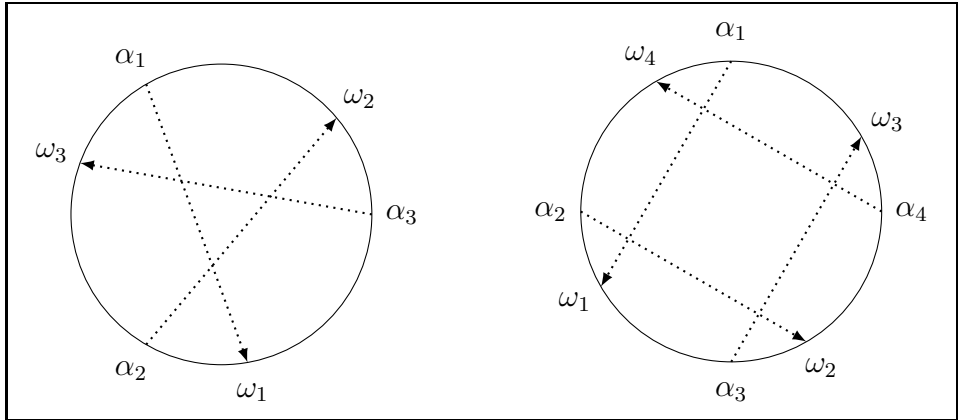
Orbits diagram for Handel's fixed point theorem: $r=3$, $r=4$
\end{center}

In Handel's original paper more general cyclic orders are allowed, but Handel's hypothesis implies the existence of a subset of the $x_{i}$'s satisfying the above hypothesis $(H_{2})$ (see the nice combinatorial argument in the introduction of the paper~\cite{MR2284059} by P. Le Calvez). Thus the original statement can be  deduced from this one. 

This theorem is a tool for detecting fixed points for homeomorphisms on surfaces, when applied to the following construction. Let $S$ be a surface without boundary, endowed with a hyperbolic metric (think of a compact surface of genus $\geq 2$, or an open subset of the sphere which is not homeomorphic to a disk 	or an annulus). Consider a homeomorphism $f:S \to S$ which is isotopic to the identity. The universal cover of $S$ is the hyperbolic disk $\bbH^2$. Lifting the isotopy, we get a homeomorphism $\tilde f: \bbH^2 \to \bbH^2$ which is a lift of $f$. One can prove that $\tilde f$ extend to a homeomorphism of the closed disk which point-wise fixes the circle boundary $\partial \bbH^2$. In this setting, every fixed point of $\tilde f$ in the open disk $\bbH^2$ projects to a fixed point of $f$.

Here is an application, due to Betsvina and Handel. In the previous construction assume $S$ is the complement of at least three points in the sphere. If $f$ has a periodic point $z$, then the trajectory of $z$ under the iterated isotopy is a closed curve $\Gamma$. If this curve is homotopic to a constant, then $z$ lifts to a periodic point of $\tilde f$, and Brouwer plane translation theorem (see below) provides a fixed point for $\tilde f$, and thus for $f$. Handel's theorem allows to get a fixed point under the weaker hypothesis that the curve is homologous to zero.
Indeed, under this hypothesis, consider the unique oriented hyperbolic geodesic $\Gamma_{0}$ of $S$ which is freely homotopic to $\Gamma$. Since $\Gamma_{0}$ is homologous to zero, its algebraic intersection number with every closed curve, and every curve joining two connected components of the complement of $S$,  is zero. Given a point $p_{0}$ in the complement of $S$, a function may be defined on the complement $U$ of $\Gamma_{0}$, assigning to a point $p$ the intersection number of $\Gamma_{0}$ with any curve from $p_{0}$ to $p$. This function is constant on each connected component of $U$, and vanishes on the complement of $S$. The maximum and the minimum of the function cannot both be zero, to fix ideas let us assume the maximum is non zero.
Consider a connected component $U_{0}$ where the function is maximal ; thus $U_{0}$ is included in $S$.
Because the function is maximal on $U_{0}$, the boundary of $U_{0}$ is made of segments of $\Gamma_{0}$ which are oriented in a coherent way. Lifting the picture to the hyperbolic plane, we find several lifts of $\Gamma_{0}$ which draw a diagram as on the above picture. To each of these lifts correspond a lift $\tilde \Gamma_{i}$ of $\Gamma$; if $\tilde z_{i}$ is a lift of $z$ on $\tilde \Gamma_{i}$, the orbits of the $\tilde z_{i}$'s satisfies the hypothesis of Handel's theorem. Thus again we get a fixed point for $f$.
For more details, and some more applications, see again the introduction of Patrice Le Calvez's paper. 

\bigskip	

One can restate the theorem by saying that when a homeomorphism of the two-disk has no fixed point in the interior, there is no family of orbits satisfying hypotheses $(H_{1})$ and $(H_{2})$. Under this viewpoint, the theorem says that orbits of a fixed point free homeomorphism of the open disk may not ``cross each others too much''. The reader may keep this idea in mind as a guideline for these notes. 

\bigskip	

We begin by recalling the classical Brouwer theory, concerning fixed point free homeomorphisms of the plane. Then we introduce and illustrate the \emph{homotopy translation arcs} which are the main objects of Handel's proof. These objects also play a central part in further developments of the theory by J. Franks and M. Handel, as in\cite{franks2003periodic} or \cite{franks2010periodic}. Finally we give the proof of the theorem.

 Handel's proof is mainly intrinsic to the interior of the disk (identified with a plane), with no reference to the boundary, and the above disk theorem follows from a plane theorem. For this intrinsic statement we follow the short exposition by S. Matsumoto (\cite{matsumoto2000arnold}). 
 A small novelty is a direct proof, using classical Brouwer theory, of the lemma which allows to deduce Handel's theorem from the intrinsic plane version (proposition~\ref{prop.translation-arcs} below). We also discuss \emph{orbit diagrams}, and propose some conjectures describing an invariant of combinatorial type associated to a finite family of orbits for a fixed point free homeomorphism of the plane, which would entail that there exists only finitely many distinct ``braid types'' for a given number of orbits.

\bigskip

Since space does not allow a proper introduction to hyperbolic geometry, we tried to avoid it as much as possible, especially in the definitions of the main objects, and tried to emphasize their purely topological aspects. 
The reader which is not familiar with this subject may read (and admit) the properties concerning hyperbolic geodesics that are listed in appendix~1 as a set of axioms. Geodesics will become essential in section 3.

\bigskip
\small
\paragraph{Acknowledgements}
I thank the organizers of the workshop, and especially Mart\'in Sambarino and Juliana Xavier, for having invited me to give these lectures. Many thanks also to Lucien Guillou and Emmanuel Militon for their careful readings of a preliminary version of these notes. Finally, I want to thank again Lucien Guillou for having introduced me to this subject, many years ago...

\normalsize

\bigskip
\bigskip
\bigskip

\section{(Classical) Brouwer theory}
\label{sec.brouwer-theory}
Handel's theorem deals with some fixed point free homeomorphisms of the open disk. By identifying the open disk with the plane, we get fixed point free homeomorphisms of the plane, which are the objects of Brouwer theory.

\subsection{Flows}
Let us recall a little bit of Poincar\'e-Bendixson theory. Let $X$ be a \textbf{non vanishing} vector field on the plane, and assume $X$ is smooth and complete, so that Cauchy-Lipschitz theorem gives rise to a flow, that is, there is a one parameter family $(\Phi^t)_{t \in \bbR}$ of diffeomorphisms of the plane tangent to the vector field: the ordinary differential equation  
$$
\frac{\partial }{\partial t} \Phi^t(x) = X(\Phi^t(x))
$$
 is satisfied. Take a smooth curve $\gamma: ]-\varepsilon, \varepsilon[ \to \bbR^2$ which is \emph{transverse} to the vector field: $\gamma'(t)$ is nowhere colinear to $X(\gamma(t))$. Then the main remark of the Poincar\'e-Bendixson theory is that \emph{no integral curve $t \mapsto \Phi^t(x)$ can meet $\gamma$ twice}. As a consequence, the map $\Psi : \bbR \times ]-\varepsilon, \varepsilon[ \to \bbR^2$ given by
$$
(x,y) \mapsto \Phi^x(\gamma(y))
$$
is one to one, and the image of $\Psi$ is an open invariant set on which the flow is conjugate to the horizontal translation flow (an \emph{invariant} flow box),
$$
\Psi((x,y) + (t,0)) = \Phi^t(\Psi(x,y)).
$$
Since no integral curve meets $\gamma$ twice, the point $\gamma(0)$ is not an accumulation point of any orbit. Thus integral curves have no accumulation point,  they tend to infinity: for every compact subset $K$ of the plane, and every $x$, the set 
$$
\{t\in \bbR, \Phi^t(x) \in K\}
$$
 is compact. We say that the map $t \mapsto \Phi^t(x)$ is \emph{proper} (or that the integral curve is a \emph{properly embedded line}).

Consider now several points $x_{1}, \dots , x_{r}$ that belongs to distinct (and thus pairwise disjoint) orbits $\Gamma_{1}, \dots \Gamma_{r}$ of the flow. We endow these orbits with the orientation induced by the flow.
The topology of any finite family of pairwise disjoint oriented properly embedded lines in the plane may be completely described by a finite invariant.
More precisely, according to the Schoenflies theorem, we can find a homeomorphism $h:\bbR^2 \to \inte(\bbD^2)$ between the plane and the open unit disk under which the image of each oriented curve $\Gamma_{i}$ becomes a chord $[\alpha_{i},\omega_{i}]$ of the unit circle. The cyclic order on the set $\{\alpha_{1}, \omega_{1}, \dots , \alpha_{r}, \omega_{r}\}$ is a total invariant of the topology of the curves, meaning that there exists a homeomorphism sending a first family of oriented curves on a second family if and only if their cyclic orders at infinity coincide. The only constraint on this cyclic order is that the chords $[\alpha_{i},\omega_{i}]$ are pairwise disjoint. For example, for two orbits there is only two possible diagrams (up to reversing the cyclic order), and five diagrams for three orbits.

\begin{center} \fbox{
\begin{tikzpicture}[scale=1]
\tikzstyle{fleche}=[>=latex,->, thick, dotted]
\draw (0,0) circle (1);
\foreach \a / \w in {150/30,-150/-30}
{
\draw [fleche] (\a:1) -- (\w:1) ;
}
\end{tikzpicture} 
\begin{tikzpicture}[scale=1]
\tikzstyle{fleche}=[>=latex,->, thick, dotted]
\draw (0,0) circle (1);
\foreach \a / \w in {150/30,-30/-150}
{
\draw [fleche] (\a:1) -- (\w:1) ;
}
\end{tikzpicture} 
}
\end{center}
\begin{center} \fbox{
\begin{tikzpicture}[scale=1]
\tikzstyle{fleche}=[>=latex,->, thick, dotted]
\draw (0,0) circle (1);
\foreach \a / \w in {150/30,180/0,-150/-30}
{
\draw [fleche] (\a:1) -- (\w:1) ;
}
\end{tikzpicture} 
\begin{tikzpicture}[scale=1]
\tikzstyle{fleche}=[>=latex,->, thick, dotted]
\draw (0,0) circle (1);
\foreach \a / \w in {150/30,180/0,-30/-150}
{
\draw [fleche] (\a:1) -- (\w:1) ;
}
\end{tikzpicture} 
\begin{tikzpicture}[scale=1]
\tikzstyle{fleche}=[>=latex,->, thick, dotted]
\draw (0,0) circle (1);
\foreach \a / \w in {150/30,0/180,-150/-30}
{
\draw [fleche] (\a:1) -- (\w:1) ;
}
\end{tikzpicture} 
\begin{tikzpicture}[scale=1]
\tikzstyle{fleche}=[>=latex,->, thick, dotted]
\draw (0,0) circle (1);
\foreach \a / \w in {-50/50,70/170,190/-70}
{
\draw [fleche] (\a:1) -- (\w:1) ;
}
\end{tikzpicture} 
\begin{tikzpicture}[scale=1]
\tikzstyle{fleche}=[>=latex,->, thick, dotted]
\draw (0,0) circle (1);
\foreach \a / \w in {50/-50,70/170,190/-70}
{
\draw [fleche] (\a:1) -- (\w:1) ;
}
\end{tikzpicture}
}\\

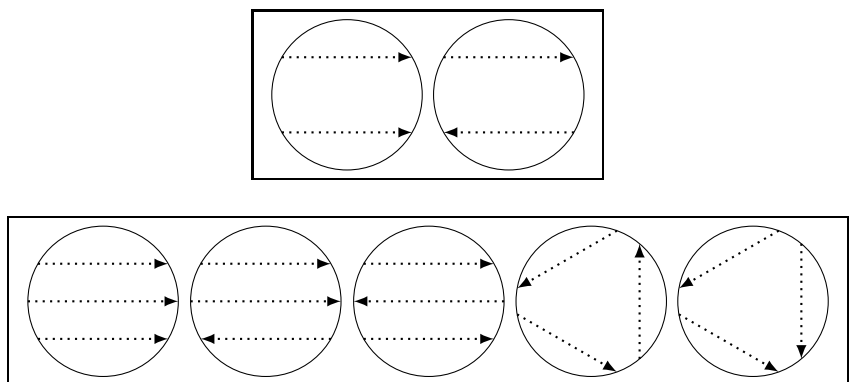
Two or three orbits of a non vanishing vector field in the plane

\end{center}

\subsection{Brouwer homeomorphisms}\label{sub.Brouwer}

A \emph{Brouwer homeomorphism} $f$ is a fixed point free, orientation preserving homeomorphism of the plane. As a consequence of Poincar\'e-Bendixson theorem, the time one map of a flow generated by a non vanishing vector field is (a special case of) a Brouwer homeomorphism, and one could say that the main purpose of Brouwer theory is to determine which properties of the planar flows generalize to general Brouwer homeomorphism. In particular, we would like to find out if there is something like a cyclic order on ends of orbits.

\bigskip

Let us first recall Brouwer \emph{plane translation theorem}, which is the analog of Poincar\'e-Bendixson theorem.
An open set $U\subset \bbR^2$ is called a \emph{translation domain} for $f$ if $U$ is the image of an embedding\footnote{A map $\Psi:X \to Y$ is called an \emph{embedding} if it is a homeomorphism between $X$ and $\Psi(X)$.} $\Psi:\bbR^2 \to \Psi(\bbR^2) = U$ such that $\Psi \circ T = f \circ \Psi$, where $T:(x,y) \mapsto (x,y)+(1,0)$ is the horizontal translation. Note that a translation domain is $f$-invariant ($f(U) = U$). Here is a weak version of  the Brouwer Plane translation theorem.
\begin{theo*}[Brouwer]
Every point of the plane belongs to a translation domain.\footnote{In the full statement, $\Psi$ my be chosen so that its restriction to every vertical straight line is proper. In what follows we will only use the weak version.}
\end{theo*}
As a corollary, exactly as in the Poincar\'e-Bendixson theory, every orbit $(f^n(x))_{n \in \bbZ}$ goes to infinity: for every compact subset $K$ of the plane, and every $x$, the set 
$$
\{n\in \bbZ, f^n(x) \in K\}
$$
 is compact. Again we will say that the orbit is proper (or \emph{locally finite}).

Under the conclusions of the theorem, let  $\Gamma$ be the image under the map $\Psi$ of any horizontal line. 
Then $\Gamma$ is an injective continuous image of the real line which is invariant under $f$, \emph{i. e.} $f(\Gamma) = \Gamma$.
Such a curve is called a \emph{streamline} for $f$, and it is an analog of the integral curves for flows. However, we must note that
\begin{itemize}
\item (non uniqueness) every point belongs to (infinitely many) distinct streamlines;
\item (non properness) although they are continuous injective images of $\bbR$, some streamlines are not properly embedded (equivalently, their images are not closed subset of the plane).
\end{itemize}
Actually the situation is the worst you can imagine: there exist examples with no properly embedded streamline, and there are uncountably many non homeomorphic possibilities even for a single non properly embedded streamline. A key point in the proof of Handel's theorem will be to replace streamlines by the more flexible ``homotopy'' streamlines.

\subsection{Translation arcs}
\label{ssec.translation-arcs}
Let $f$ be a Brouwer homeomorphism.
A simple arc\footnote{A \emph{simple arc}, sometimes just called an arc, is an injective  continuous image of $[0,1]$.}   $\alpha$ satisfying 
\begin{enumerate}
\item $\alpha(1) = f(\alpha(0))$,
\item $\alpha \cap f(\alpha) = \{\alpha(1)\}$
\end{enumerate}
is called a \emph{translation arc} for the point $\alpha(0)$.

\begin{center}
\fbox{
\begin{tikzpicture}[scale=1]
\draw  (0,0) node {$\bullet$} ..controls +(0.5,0.5) and +(-0.5,-0.5)  .. (1,0)
		node {$\bullet$} 
		..controls +(0.5,-0.5) and +(-0.5,-0.5)  .. (2,0)
		node {$\bullet$} ;
\draw (0.5,0.5) node {$\alpha$} ;
\draw (1.5,-0.8) node {$f(\alpha)$} ;

\end{tikzpicture}~ 
}
\\
A translation arc
\end{center}

The following is a fundamental lemma of the theory. For a proof see for example~\cite{barge1993recurrent,guillou1994Brouwer}.
\begin{lemm}(``Free disk lemma'')
Let $D\subset \bbR^2$ be a topological disk, \emph{i. e.} a set homeomorphic either to the open or to the closed unit disk. Assume that $D$ is \emph{free}, that is, $f(D) \cap D = \emptyset$. Then $f^n(D) \cap D = \emptyset$ for every $n \neq 0$. 
\end{lemm}
Note that a small enough disk centered at any point $x$ is free, thus the lemma incorporates the fact that no point is periodic.
The following corollary implies that the union of all the iterates of a translation arc is a streamline.
\begin{coro}If $\alpha$ is a translation arc then
$f^{n}(\alpha) \cap \alpha \neq \emptyset$ if and only if $n =-1,0,1$.
\end{coro}
\begin{proof}
Assume by contradiction that there is $x \in \alpha$ such that $f^{-n}(x) \in \alpha$ for some $n \neq -1,0,1$.
A special case is when $\{x,f^{-n}(x)\} = \{\alpha(0),\alpha(1)\}$. Then $x$ must be a periodic point, which contradicts the lemma.
Thus the special case does not occur, which means that the sub-arc of $\alpha$ joining $x$ and $f^{-n}(x)$ is a not equal to $\alpha$.
In particular it is free, and thus by thickening it, we see that it is included in a free topological open disk. This disk contradicts the free disk lemma.
\end{proof}
The above proof implicitly uses the following version of the Schoenflies theorem (where?...): \emph{any simple arc of the plane is the image of a segment under a homeomorphism of the plane}.

\bigskip

We end this section by giving a direct construction of a translation arc (with no reference to the plane translation theorem).
A topological closed disk $B$ is called \emph{critical} if the interior of $B$ is free, but $B$ is not. Let $B$ be a critical disk containing some point $x$ in its interior. Choose some point $y \in B \cap f(B)$, some arc $\gamma_{1}$ joining $x$ to $y$ and included in $\inte B$ except at $y$, and some arc $\gamma_{2}$ joining $x$ to $f^{-1}(y)$ and included in $\inte B$ except at $f^{-1}(y)$, such that $\gamma_{1} \cap \gamma_{2} = \{x\}$, so that $\gamma_{1} \cup \gamma_{2}$ is a simple arc. We construct a simple arc from $x$ to $f(x)$ by gluing $\gamma_{1}$ with $f(\gamma_{2})$.
The following statement implies that such an arc is a translation arc for $x$.

\begin{center}
\fbox{
\begin{tikzpicture}[scale=1.4]
\draw (0,0)  circle (1cm);

\draw  (45:1) .. controls +(-1,1) and +(0.5,2)  .. (3,0)
		.. controls +(-0.5,-2) and +(-1,-1)  .. (-45:1)
		 .. controls +(1,1) and +(0.5,-0.5) .. (1.6,0.5)
		.. controls +(-0.5,0.5) and  +(0.5,-0.5) .. (45:1) ;
\draw (0,0) node {$\bullet$} node [below] {$x$} -- (45:1) node [pos=0.4,above] {$\gamma_{1}$}  node {$\bullet$} node [below] {$y$} ;
\draw (45:1)  .. controls +(30:1) and +(0,1)  .. (2,0) node {$\bullet$} node [below] {$f(x)$} ;
\draw [dashed] (0,0) -- (135:1) node [pos=0.6,below] {$\gamma_{2}$} node {$\bullet$} node [above,left] {$f^{-1}(y)$} ;
\draw (-135:1.5) node {$B$};
\draw (2.5,-1.5) node {$f(B)$};
\end{tikzpicture}~ 
}
\\
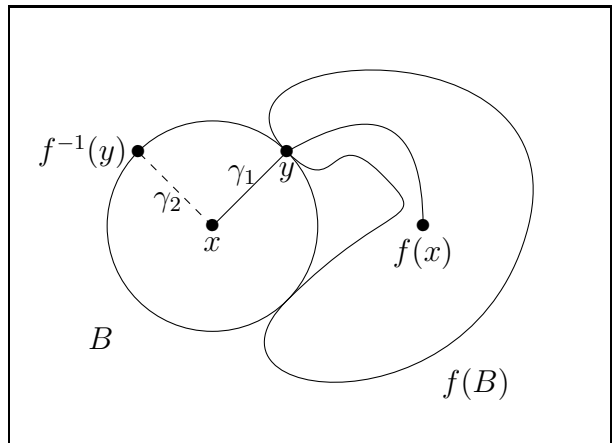
A critical disk and a geometric translation arc
\end{center}

\begin{coro}[critical disks]\label{coro.critical}
$f^{n}(B) \cap B \neq \emptyset$ if and only if $n =-1,0,1$.
\end{coro}

By making a euclidean disk grow until it touches its image, we see that for every given point there is a unique critical disk among euclidean disks centered at the point. When $B$ is a euclidean disk, we may choose $\gamma_{1}$ and $\gamma_{2}$ to be euclidean segments in the previous construction, as on the previous figure. Then we say that the translation arc is \emph{geometric}.

\begin{proof}[Proof of corollary~\ref{coro.critical}]
Use again the idea of the proof of the corollary on translation arcs. Details are left to the reader. 
\end{proof}

\begin{exercise}
Prove that any neighborhood of any arc $\gamma$ joining a point $x$ to its image contains a topological disk which is critical and contains $x$ in its interior.
\end{exercise}

\subsection{The homotopy class of translation arcs}

The following proposition will be the key to deduce the fixed point theorem, as stated in the introduction, from an ``intrinsic'' theorem dealing with Brouwer homeomorphisms.
It is a weak version of Corollary~6.3 of~\cite{handel1999fixed}. The weak version is sufficient for our needs, but we will also be able to deduce the strong version from the weak (corollary~\ref{coro.strong} below).

Let $\cO(x_{0}) = \{f^n(x_{0}), n \in \bbZ\}$. Let $\alpha, \alpha':[0,1] \to \bbR^2$ be two curves joining $x_{0}$ to $f(x_{0})$. A \emph{homotopy} (with fixed end-points) from $\alpha$ to $\alpha'$ is a continuous map $H:[0,1]^2 \to \bbR^2, (s,t) \mapsto \alpha_{t}(s)$ such that $\alpha_{0}=\alpha, \alpha_{1}=\alpha'$ and each $\alpha_{t}$ is a curve joining $x_{0}$ to $f(x_{0})$. The homotopy is \emph{relative to $\cO(x_{0})$} if every curve $\alpha_{t}$ meets $\cO(x_{0})$ only at its end-points. The homotopy is an \emph{isotopy} if every curve $\alpha_{t}$ is injective. Standard results in surface topology imply that two injective curves $\alpha,\alpha'$ which are homotopic relative to $\cO(x_{0})$ are also isotopic relative to $\cO(x_{0})$.\footnote{This fact is actually included in the properties of hyperbolic geodesics, see in particular property 3 of the appendix.}

\begin{prop}\label{prop.translation-arcs}
Let $\alpha_{0}, \alpha_{1}$ be two translation arcs for a Brouwer homeomorphism $f$ for the same point $x_{0}$. Then $\alpha_{0}$ and $\alpha_{1}$ are homotopic relative to $\cO(x_{0})$.
\end{prop}

\begin{exercise}
In the special case when $f$ is a translation, one may consider the quotient $\bbR^2/f$, which is an infinite annulus. What can you say about the image of a translation arc in the quotient? The proposition, in this special case, should become ``obvious''. 
\end{exercise}

To prove the proposition we need two lemmas. The first lemma says that, up to conjugacy,  geometric translation arcs have nothing special.
\begin{lemm}\label{lem.all-geometric}
For every translation arc $\alpha$ there exists a homeomorphism $g$ isotopic to the identity such that the arc $g(\alpha)$ is a geometrical translation arc for $gfg^{-1}$. We may further assume that $g(\alpha)$ joins $0$ to $1$.
\end{lemm}
We insist that $g$ will be isotopic to the identity, but not isotopic to the identity relative to an orbit of $f$. 
\begin{proof}
We look for a situation homeomorphic to the picture of a geometrical arc and its critical euclidean disk. Namely, we want to find a
topological closed disk $B$ which is critical, contains $\alpha(0)$ in its interior, and such that 
the boundary $\partial B$ meets $f^{-1}(\alpha) \cup \alpha$ in exactly two points, a point $y$ on $\alpha$ and its inverse image $f^{-1}(y)$ on $f^{-1}(\alpha)$.
 Once we have found such a $B$, an adapted version of Schoenflies theorem provides a homeomorphism $g$ such that $g(B)$ is a euclidean disk centered at $g(x)$, and $g(f^{-1}(\alpha) \cap B), g(\alpha \cap B)$ are euclidean segments, and such a $g$ satisfies the conclusion of the lemma.
Here is a way to construct $B$. Up to making a first conjugacy, one can assume that $\alpha$ and its inverse image are horizontal segments. Choose a vertical small segment $\gamma$ centered at the middle of $\alpha$, such that $f^{-1}(\gamma)$ is disjoint from $\gamma$. Up to a new conjugacy, we assume that $f^{-1}(\gamma)$ is also a vertical segment. Then $B$ may be chosen as a thin horizontal ellipse (or rectangle) tangent to $\gamma$ and $f^{-1}(\gamma)$, as shown on the picture.
\end{proof}

\begin{center}
\fbox{
\begin{tikzpicture}[scale=2]
\draw  (0,0) node {$\bullet$} -- (1,0)
		node {$\bullet$}  -- (2,0)
		node {$\bullet$} ;
\draw (0,0) node [below] {$f^{-1}(\alpha)$} ;
\draw (1.7,0) node [below] {$\alpha$} ;
\draw (0.5,-0.5)  -- (0.5,0.5);
\draw (1.5,-0.5)  -- (1.5,0.5);
\draw (0.5,0.5) node [above] {$f^{-1}(\gamma)$} ;
\draw (1.5,0.5) node [above] {$\gamma$} ;

\pgfsetfillopacity{0.1} 
\draw [fill] (1,0) ellipse (0.5 and 0.1);
\end{tikzpicture}~ 
}
\\
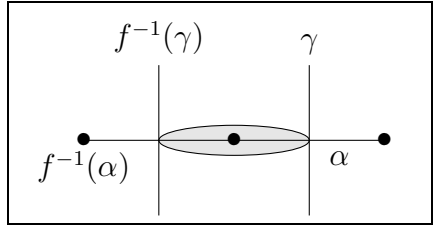
Construction of an adapted critical disk
\end{center}

As before we consider a Brouwer homeomorphism $f$, and some point $x_{0}$. To prove the proposition we need to thoroughly analyze the geometric translation arcs at $x_{0}$.
Let $B_{f}$ be the unique euclidean critical disk centered at $x_{0}$, and $S$ be its circle boundary. In the (easy) case when $B_{f}$ meets its image at a single point, there is a unique geometric translation arc for $x_{0}$, and we define $C$ to be this arc (as this is the easy case, we will not discuss it anymore). Assume we are in the opposite case (see the picture below). 
The set $S \setminus f(B_{f})$ is a union of open arcs of the circle $S$;  exactly one of these arcs is included in the boundary of the unbounded component of $\bbR^2 \setminus (B_{f} \cup f(B_{f}))$, let $y,z$ be the (distinct) end-points of this arc. Let $\gamma_{y}, \gamma_{z}$ be the geometric translation arcs containing respectively $y,z$. It is easy to see that $\gamma_{y} \cup \gamma_{z}$ is a Jordan curve, let $C$ be the closed topological disk bounded by this curve.
We claim that 
$$
C \cap f^{-1}(C) = \{x_{0}\}. 
$$
Indeed, we first note that $\partial C \cap f^{-1}(\partial C) = \{x_{0}\}$: this is because 

\noindent (1) by construction of geometric translation arcs, $\partial C \cap f^{-1}(\partial C) \cap B_{f} = \{x_{0}\}$; 

\noindent (2) $\partial C \setminus B_{f} \subset f(B_{f})$, while $f^{-1}(\partial C) \setminus B_{f} \subset f^{-1}(B_{f})$, and $f(B_{f}) \cap f^{-1}(B_{f})=\emptyset$ (corollary~\ref{coro.critical} on critical disks). 

From this we deduce that either the claim holds, or one of the two disks $C$ and $f^{-1}(C)$ contains the other one, but in this last case the Brouwer fixed point theorem would provide a fixed point for $f$, a contradiction. This proves the claim.

\begin{center}
\fbox{
\begin{tikzpicture}[scale=1.4]
\draw (0,0)  circle (1cm);
\draw (-1,-1) node {$B_{f}$};
\draw (3,-1) node {$f(B_{f})$};
\draw  (45:1) .. controls +(-1,1) and +(0.5,2)  .. (3,0)
		.. controls +(-0.5,-2) and +(-1,-1)  .. (-45:1)
		 .. controls +(1,1) and +(0.5,-0.5) .. (1.6,0.5)
		.. controls +(-0.5,0.5) and  +(0.5,-0.5) .. (45:1) ;
\draw [thick] (0,0) node {$\bullet$} node [left] {$x_{0}$} -- (45:1) 
  node {$\bullet$} node [below] {$y$} ;
\draw [thick]  (45:1)  .. controls +(30:1) and +(0,1)  .. (2,0) node {$\bullet$} node [right] {$f(x_{0})$} ;
\draw  [thick] (0,0)  -- (-45:1) 
  node {$\bullet$} node [above] {$z$} ;
\draw  [thick]  (2,0)  .. controls +(0,-1) and +(1,0) ..  (-45:1) ;

\pgfsetfillopacity{0.05} 
\def \eps{0.1}
\path [draw,dashed,fill] (0-\eps,0) -- (0.7,0.7+\eps) .. controls +(30:1) and +(0,1) .. (2+\eps,0)  .. controls +(0,-1) and +(1,0) ..  (0.7,-0.7-\eps)  -- cycle ;

\end{tikzpicture}~ 
}
\\ \nopagebreak
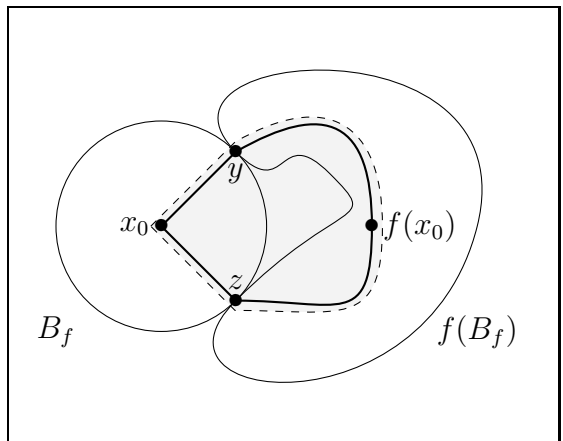
Construction of the topological disks $C$ and $D$
\end{center}

Note that $C$ contains all the geometric translation arcs for the point $x_{0}$. From the disk $C$ we will construct another (slightly bigger) disk $D$ whose properties are given by the following lemma.
\begin{lemm}\label{lem.local}
There exists a topological disk $D$, and a neighbourhood $V$ of $f$ in the space of Brouwer homeomorphisms (equipped with the topology of uniform convergence on compact subsets of the plane), such that
for every $f'\in V$,
\begin{itemize}
\item $\inte D$ contains every geometric translation arc for the map $f'$ and the point $x_{0}$,
\item $D \cap \{f'^n(x_{0}), n \in \bbZ\} = \inte D \cap \{f'^n(x_{0}), n \in \bbZ\} =  \{x_{0}, f'(x_{0})\}$.
\end{itemize}
\end{lemm}

\begin{proof}
We treat only the case when $B_{f} \cap f(B_{f})$ is not reduced to a single point (the opposite case is similar but far easier).
First assume that $D$ is any topological disk whose interior contains $C$. In particular, the interior of $D$ contains all the geometric translation arcs for $f$. Consider another Brouwer homeomorphism $f'$, and let $B_{f'}$ be the unique euclidean critical disk for $f'$ which is centered at $x_{0}$. When $f'$ is close to $f$, the disk $B_{f'}$ is close to $B_{f}$, and the intersection $B_{f} \cap f(B_{f})$ must be included in $D$. From this it is not difficult to see that $D$ contains all the geometric translation arcs for $f'$, which gives the first property of the lemma.

It remains to get the second property.
For this we choose $D$ to be a small enough neighborhood of $C$, so that it may be written as the union of two topological closed disks $\delta,\delta'$ satisfying the following conditions:
\begin{enumerate}
\item $\delta,\delta'$ are free for $f$;
\item $x_{0} \in \inte \delta$ and $f(x_{0}) \in \inte \delta'$. 
\end{enumerate}
This is possible since the disk $C$ is ``almost free'' (see the above claim, $C \cap f^{-1}(C) = \{x_{0}\}$).
 Define $V$ to be the set of Brouwer homeomorphisms $f'$ such that the first property of the lemma holds, and such that the above conditions 1 and 2 are still satisfied for $f'$, namely   $\delta,\delta'$ are free for $f'$, and  $x_{0} \in \inte \delta$ and $f'(x_{0}) \in \inte \delta'$. 
It is clear that $V$ is a neighborhood of $f$.
Let $f' \in V$, and let us check the second property of the lemma. Consider an integer $n$ such that $f'^n(x_{0})$ belongs to  $D$. Since $x_{0}$ is in $\delta$ which is free for $f'$, $f'^n(x_{0})$ may not be in $\delta$ unless $n=0$ (free disk lemma). Likewise, since $f'(x_{0})$ is in $\delta'$ which is free for $f'$, $f'^n(x_{0})$ may not be in $\delta'$ unless $n=1$. Thus $f'^n(x_{0})$ cannot be in $D = \delta \cup \delta'$ unless $n=0,1$, as wanted.
\end{proof}

\begin{proof}[Proof of the proposition]
Lemma~\ref{lem.all-geometric} provides $g_{0}, g_{1}$ such that $g_{0}\alpha_{0}, g_{1} \alpha_{1}$ are geometric translation arcs respectively for $g_{0} f g_{0}^{-1}, g_{1} f g_{1}^{-1}$. The space $\homeo_{0}(\bbR^2)$ of homeomorphisms of the plane that are isotopic to the identity is arcwise connected, thus there exists a continuous path $(g_{t})_{t \in [0,1]}$ in that space joining $g_{0}$ and $g_{1}$. Let $f_{t} = g_{t} f g_{t}^{-1}$. By composing $g_{t}$ with the translation that sends $g_{t}(x_{0})$ to $0$, we may assume that $g_{t}(x_{0}) = 0$ for every $t$ ; likewise, by composing $g_{t}$ with the (unique) complex multiplication that sends $g_{t} f (x_{0})$ to $1$, we may assume that $f_{t}(0) = g_{t}(f(x_{0})) = 1$ for every $t$.
Since complex affine maps preserves segments and circles, these modifications of $g_{t}$ do not alter the previous properties: $g_{0}\alpha_{0}, g_{1}\alpha_{1}$ are still \emph{geometric} translation arcs.

 We consider the disks $D(f_{t})$ and the neighbourhoods $V(f_{t})$ provided by Lemma~\ref{lem.local}, applied at the point $x_{0}=0$.
By compactness, there exist $t_{0} = 0 < ... < t_{\ell} = 1$ such that every $f_{t}$ with $t \in [t_{i}, t_{i+1}]$ is included in some $V_{i} = V(f_{t'_{i}})$. Let $D_{i} = D(f_{t'_{i}})$. Let $\alpha'_{0} = g_{0}\alpha_{0}, \alpha'_{1} = g_{1} \alpha_{1}$ which are geometric translation arcs resp. for $f_{t_{0}}, f_{t_{\ell}}$, and for every $i=1, \dots , \ell-1$ choose some geometric translation arc $\alpha'_{t_{i}}$ for $f_{t_{i}}$.
Now for $i=0, \dots, \ell-1$, both arcs $\alpha'_{t_{i}}$ and $\alpha'_{t_{i+1}}$ go from the point $0$ to the point $1$ and are contained in the disk $D_{i}$ (first point of lemma~\ref{lem.local}). Thus they are homotopic within $D_{i}$: we may find a continuous family of curves $(\alpha'_{t})_{t \in [t_{i},t_{i+1}]}$ connecting both arcs and still going from $0$ to $1$ and included in $D_{i}$. Point two of lemma~\ref{lem.local}  entails that for every $t$,
$$
\alpha'_{t}\cap \{f_{t}^n(0), n \in \bbZ\} =  \{0,1\}.
$$
Now let $\alpha_{t} = g_{t}^{-1} \alpha'_{t}$. This is a homotopy from $\alpha_{0}$ to $\alpha_{1}$ relative to the orbit $\cO(x_{0})$\footnote{This proof was sketched in the author's PhD thesis.}.
\end{proof}

\begin{exercise}
Prove the weak version of the plane translation theorem as a consequence of the free disk lemma. Hints: by thickening a translation arc we may construct a critical disk $\delta$ such that $\delta \cap f(\delta)$ is a simple arc. The iterates of $\delta$ give rise to a translation domain.
\end{exercise}

\bigskip
\bigskip
\bigskip

\section{Homotopy translation arcs}
\label{sec.hta}

Integral curves of flows never cross each other.
We would like to know to what extent  orbits of a Brouwer homeomorphism can cross each other, but it is not easy to give a precise meaning to this. In this direction we have defined translation arcs, in order to replace integral curves of flows by the union of iterates of a translation arc, also called streamlines. But for a general Brouwer homeomorphism the topology of a streamline may be complicated. Thus streamlines are not appropriate to define a notion of crossing. 
The idea is to relax the invariance to a \emph{homotopy} invariance. 

\bigskip

%

In this section $f$ is an orientation preserving homeomorphism of the plane.
We select finitely many points $x_{1}, \dots x_{r}$ with disjoint orbits, and let 
$$
\cO = \cO(x_{1}, \dots x_{r}) = \{f^n(x_{i}), n \in \bbZ, i = 1, \dots, r\}
$$
denote the union of their orbits. We do not demand that $f$ is fixed point free, but the $x_{i}$'s are assumed to have proper orbits: in other words they are not periodic and the set $\cO$ is locally finite. The plane translation theorem tells us that this is automatic if $f$ is fixed point free.

\subsection{Definitions}
\label{ss.hta-def}

We consider the continuous  curves $\alpha: [0,1] \to \bbR^2$ joining two points $x,y \in \cO$, whose interior $\inte \alpha := \alpha((0,1))$ is disjoint from $\cO$, and whose restriction to $(0,1)$ is injective (thus the image of $\alpha$ is homeomorphic to the circle or to the closed interval). Such a curve is said to be \emph{inessential} if $\alpha(0)=\alpha(1)$ and the bounded component of $\bbR^2 \setminus \alpha$ does not contain any point of $\cO$, otherwise it is called \emph{essential}. Let $\cA$ be the set of essential curves. The set $\cA$ is endowed with the topology of uniform convergence, and the connected components of $\cA$ are called \emph{homotopy classes}\footnote{Remember that, in this context, the notions of isotopy and homotopy coincide.}. Two curves in the same homotopy class will be said to be \emph{homotopic relative to $\cO$}. The homotopy class of $\alpha$ will be denoted by $\underline{\alpha}$, and the set of homotopy classes by $\underline{\cA}$. Note that the end-points $\underline{\alpha}(0)$, $\underline{\alpha}(1)$ are well-defined (homotopic curves have the same end-points). The map $f$ induces a map $\underline{f}$ on homotopy classes.

We will say that two curves $\alpha,\beta \in \cA$ are \emph{homotopically disjoint}, and write $\underline{\alpha} \cap \underline{\beta} = \emptyset$, if $\underline{\alpha} \neq \underline{\beta}$ and there exist $\alpha' \in\underline{\alpha}, \beta' \in \underline{\beta}$ such that $\alpha' \cap \beta' \subset \cO$, that is, the curves are disjoint except maybe at their end-points.
Let $\alpha \in \cA$ be a simple arc joining some $x\in X$ to its image $f(x)$. 
The arc is a \emph{homotopy translation arc} (for the point $x$) if the curve is homotopically disjoint from all its iterates, that is, for every $n \neq 0$,
$\underline{f}^n(\underline{\alpha}) \cap \underline{\alpha} = \emptyset$. 
A sequence of curves $(\alpha_{n})_{n \geq 0}$ in $\cA$ is said to be \emph{homotopically proper} if 
 for every compact subset $K$ of the plane, there exists $n_{0}$ such that for every $n \geq n_{0}$, there exists $\alpha' \in \underline{\alpha}_{n}$ such that $\alpha' \cap K = \emptyset$. 
A homotopy translation arc $\alpha$ is \emph{forward proper} if the sequence $(f^n(\alpha))_{n\geq0}$ is homotopically proper. The notion of \emph{backward proper homotopy translation arc} is defined in a symmetric way.

\begin{exercise}
Prove that a sequence  $(\alpha_{n})_{n \geq 0}$ is homotopically proper if and only if 
 for every $\beta \in \cA$, for every $n$ large enough, $\underline{\alpha}_{n} \cap \underline{\beta} = \emptyset$.
Hints: for the difficult part hyperbolic geodesics make life easier (see the appendix).
\end{exercise}

\subsection{Examples}
\label{ssec.examples}
The following exercise explores the properties of homotopy translation arcs on the easiest examples. 

\newpage
\begin{exercise}~

\subexo The pictures on the next page show examples of orbits of some fixed point free homeomorphisms. We start with the first three examples, which are time one maps of flows.
Try to draw several distinct homotopy translation arcs for the same point. Are they backward or forward proper?

\noindent \begin{minipage}[b]{0.6\linewidth}
\subexo  We now consider a more involved example. We begin with a map $f$ which is the time one map of a flow, with five trajectories as on the picture on the right.
The wanted map $f'=\varphi \circ f$ is obtained as the composition of $f$ with a map $\varphi$ supported on a disk $\delta$ that is free for $f$ (the shaded disk on the picture).
\end{minipage}
\noindent \begin{minipage}[b]{0.4\linewidth}
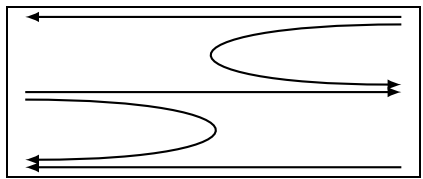
\begin{center}
\fbox{
\begin{tikzpicture}[scale=0.5]
\tikzstyle{fleche}=[>=latex,->, thick]
\draw [fleche] (0,0) -- (10,0) ; 
\draw [fleche] (10,2) -- (0,2) ;
\draw [fleche] (10,-2) -- (0,-2) ;
\def \control {6.7}
\draw [fleche] (10,1.8) .. controls +(left:\control) and +(left:\control) .. (10,0.2) ;
\draw [fleche] (0,-0.2) .. controls +(right:\control) and +(right:\control) .. (0,-1.8) ;
\end{tikzpicture} 
}
\end{center}
\end{minipage}
Let $x_{1}$ be a point, and $\gamma$ be the translation arc for $x_{1}$ as depicted on the figure. What are the $f'$ iterates of $x_{1}$?
Draw the iterates of $\gamma$. Is the corresponding homotopy translation arc forward proper? Does $x_{1}$ admit a forward proper homotopy translation arc?
Is there any homotopy translation arc for $x_{0}$ which is both forward and backward proper?
\end{exercise}

\begin{exercise}
Find a situation with $r=2$ and a homotopy translation arc which is not homotopic to a (classical) translation arc. Hints: consider a flow with several parallel Reeb strips.
\end{exercise}

\begin{figure}[p]
\begin{center}
\fbox{
\begin{tikzpicture}[scale=1]
\tikzstyle{fleche}=[>=latex,->, thick]
\draw [fleche] (0,0) -- (10,0) ; 
\draw [fleche] (0,2) -- (10,2) ;
\foreach \k in {1,3,...,9} 
	{\draw (\k,0) node {$\bullet$} ; \draw (\k,2) node {$\bullet$} ; }
\draw (5,-0.4) node {$x_{2}$} ;
\draw (5,2.4) node {$x_{1}$} ;
\end{tikzpicture} 
}
\\
$f$ is a translation, $r=2$
\end{center}

\bigskip

\begin{center}
\fbox{
\begin{tikzpicture}[scale=1]
\tikzstyle{fleche}=[>=latex,->, thick]
\draw [fleche] (0,0) -- (10,0) ; 
\draw [fleche] (10,2) -- (0,2) ;
\draw (5,-0.4) node {$x_{2}$} ;
\draw (5,2.4) node {$x_{1}$} ;
\foreach \k in {1,3,...,9} 
	{\draw (\k,0) node {$\bullet$} ; \draw (\k,2) node {$\bullet$} ; }

\end{tikzpicture} 
}
\\
$f$ is the time one map of the Reeb flow, $r=2$
\end{center}

\bigskip

\begin{center}
\fbox{
\begin{tikzpicture}[scale=1]
\tikzstyle{fleche}=[>=latex,->, thick]
\draw [fleche] (0,0) -- (10,0) ; 
\draw [fleche] (10,2) -- (0,2) ;
\foreach \k in {1,3,...,9} 
	{\draw (\k,0) node {$\bullet$} ; \draw (\k,2) node {$\bullet$} ; }
\draw [fleche] (0,0.2) .. controls +(right:5) and +(right:5) .. (0,1.8) 
\foreach \p in {0.1,0.3,...,0.9} {node[pos=\p]{$\bullet$}} node[pos=0.5, right] {$x_{3}$} ;

\draw (5,-0.4) node {$x_{2}$} ;
\draw (5,2.4) node {$x_{1}$} ;
\end{tikzpicture} 
}
\\
$f$ is the Reeb flow, $r=3$
\end{center}

\bigskip

\begin{center}
\fbox{
\begin{tikzpicture}[scale=1]
\tikzstyle{fleche}=[>=latex,->, thick]
\draw [fleche] (0,0) -- (10,0) ; 
\draw [fleche] (10,2) -- (0,2) ;
\draw [fleche] (10,-2) -- (0,-2) ;
\draw (9,1.4) node {$x_{1}$};
\draw (8,1.4) node {$\gamma$};
\foreach \k in {1,3,...,9} 
	{\draw (\k-1,0) node {$\bullet$} ; \draw (\k,2) node {$\bullet$} ; \draw (\k,-2) node {$\bullet$} ;}
\def \control {6.7}
\draw [fleche] (10,1.8) .. controls +(left:\control) and +(left:\control) .. (10,0.2) 
\foreach \p in {0.05,0.2,0.5,0.8,0.95} {node[pos=\p]{$\bullet$}}; 

\draw [fleche] (0,-0.2) .. controls +(right:\control) and +(right:\control) .. (0,-1.8) 
\foreach \p in {0.05,0.2,0.5,0.8,0.95} {node[pos=\p]{$\bullet$}}; 
\pgfsetfillopacity{0.1} 
\filldraw   (5,0) ellipse (0.2 and 1.2);

\draw (0,-2.5) node {} ;
\draw (0,2.5) node {} ;
\end{tikzpicture} 
}
\hspace{0.5cm}
\fbox{
\begin{tikzpicture}[scale=1.5]
\tikzstyle{fleche}=[>=latex,->, thick]
\draw (5,1) node {$\bullet$} ;
\draw (5,-1) node {$\bullet$} ;
\draw (5.5,0) node {$\varphi$} ;
\draw [fleche] (5,0.8) -- (5,-0.8) ;
\pgfsetfillopacity{0.1} 
\filldraw   (5,0) ellipse (0.2 and 1.4);
\end{tikzpicture} 
}
\\
$f'=\phi \circ f$ is the composition of a flow and a small perturbation, $r=4$
\end{center}

\end{figure}

\subsection{Backward \emph{and} forward proper homotopy translation arcs}

The following theorem describes the situation up to $r=3$\footnote{The results in this section will not be used in the proof of the fixed point theorem.}.
\begin{theo*}
Assume $f$ is fixed point free, in other words $f$ is a Brouwer homeomorphism. If $r=1,2,3$ then there exists  homotopy translation arcs $\gamma_{1},\dots \gamma_{r}$ for the points $x_{1}, \dots, x_{r}$ that are both backward and forward proper and such that, for every $n\in \bbZ$ and every $i \neq j$, the arcs $\gamma_{i}$ and $f^n(\gamma_{j})$ are homotopically disjoint.
\end{theo*}
The cases $r=1,2$ are in \cite{handel1999fixed} (they are essentially equivalent to Theorem 2.2 and 2.6 of that paper). The case $r=3$ may be proved using Handel's techniques. The last example in the previous section shows that the statement becomes false when $r\geq 4$. 
We will only provide a proof in the case $r=1$, using proposition~\ref{prop.translation-arcs} about the uniqueness of homotopy class of translation arcs, and the construction of translation arcs using critical disks. The other cases are much harder, and necessitate the concepts of reducing lines and fitted families that disappear under iteration, see~\cite{handel1999fixed}.

\begin{proof}[Proof when $r=1$]
We consider a Brouwer homeomorphism $f$.
As a preliminary we prove that a point sufficiently near infinity admits a translation arc sufficiently near infinity. More precisely,  let $K$ be a compact subset of the plane, and $C$ be a large disk containing both $K$ and $f^{-1}(K)$.
Let $x$ be a point outside $C \cup f^{-1}(C)$. Then $f(x)$ is outside $C$.
According to the exercise at the end of section~\ref{sec.brouwer-theory}.\ref{ssec.translation-arcs}, since the complement of $C$ is arcwise connected, we may find a topological disk $B$ containing $x$ in its interior which is critical. We have seen that there exists  a translation arc $\gamma$ for $x$ included in $B \cap f(B)$. By choice of $C$ the arc $\gamma$ is disjoint from $K$.

Now consider the orbit $\cO = \cO(x_{1})$ of some point $x_{1}$.
Let  $\gamma_{0}$ be any (classical) translation arc for the point $x_{1}$. Of course $\gamma_{0}$ is a homotopy translation arc, let us prove that it is forward proper.
Let $K$ be a compact subset of the plane. The forward orbit of $x_{1}$ is going to infinity, and according to the preliminary property, for every $n$ large enough there exists a translation arc $\gamma_{n}$ for $f^n(x_{1})$ which is disjoint from $K$. According to the uniqueness of homotopy class of translation arcs, the arc $f^n(\gamma_{0})$ is homotopic to $\gamma_{n}$ relative to $\cO$. Thus $\gamma_{0}$ is forward proper. Similarly, it is backward proper.
\end{proof}

As a corollary, we obtain that homotopy translation arcs are essentially unique when $r=1$. This reinforces proposition~\ref{prop.translation-arcs}.
\begin{coro}[\cite{handel1999fixed}, corollary 6.3]
\label{coro.strong}
Let $f$ be a Brouwer homeomorphism and $\cO = \cO(x_{1})$ be the orbit of some point $x_{1}$. Let $\gamma, \gamma'$ be two homotopy translation arcs for $x_{1}$. Then they are homotopic relative to $\cO$.
\end{coro}
The proof is mainly an excuse to begin to play with  the family $\cH$ of hyperbolic geodesics.
We refer to the properties of $\cH$ as listed in the appendix. 

\begin{proof}
Let $\gamma_{0}$ be a homotopy translation arc for $x_{1}$ which is both backward and forward proper, as given by the case $r=1$ of the theorem.
For every $n$, we denote by $f_{\sharp}^n\gamma_{0}$ the unique geodesic homotopic to $f^n(\gamma_{0})$ relative to $\cO$ (property 1 of the appendix). 
For $p \neq q$ the geodesics $f_{\sharp}^p\gamma_{0}$ and $f_{\sharp}^q\gamma_{0}$ are in minimal position (property 2), and since $\gamma_{0}$ is a homotopy translation arc, they must be disjoint (except possibly at their end-points). Since the sequence $(f^n(\gamma_{0}))$ is homotopically proper, the sequence $(f_{\sharp}^n(\gamma_{0}))$ of corresponding geodesics is proper (property 5).
Thus the union
$$
\bigcup_{n \in \bbZ} f_{\sharp}^n(\gamma_{0})
$$
is a properly embedded line: there exists a homeomorphism $\Phi \in \homeo_{0}(\bbR^2)$ sending this line to $\bbR \times \{0\}$, and more precisely we may choose $\Phi$ such that $\Phi(f_{\sharp}^n(\gamma_{0})) = [n,n+1] \times \{0\}$ for every $n \in \bbZ$. Since our problem is invariant under conjugacy, up to replacing $f$ and $x_{1}$ by $\Phi f\Phi^{-1}$ and $\Phi(x_{1})$ (and the family $\cH$ of geodesics by $\Phi(\cH)$), we may assume that $x_{1} = (0,0)$ and $f_{\sharp}^n(\gamma_{0}) = [n,n+1] \times \{0\}$ for every $n$. From now on we work with these hypotheses.

Consider the family $\{f([n,n+1] \times \{0\}), n \in \bbZ\}$. It is locally finite, and its elements are pairwise non-homotopic and disjoint. According to property 3 of the appendix, there exists some $\Phi \in \homeo_{0}(\bbR^2, \cO)$ sending each element of this family to a geodesic (where  $\homeo_{0}(\bbR^2, \cO)$ denotes the identity component in the space of homeomorphism of the plane that pointwise fixe $\cO$; we say that elements of $\homeo_{0}(\bbR^2, \cO)$ are \emph{isotopic to the identity relative to $\cO$}). The curve $f([n,n+1] \times \{0\})$ is homotopic to the geodesic $[n+1,n+2] \times \{0\}$ relative to $\cO$, and since $\Phi$ is isotopic to the identity relative to $\cO$, so is the curve $\Phi f([n,n+1] \times \{0\})$. By uniqueness of the geodesic in a given homotopy class, we deduce that 
$\Phi(f([n,n+1] \times \{0\})) = [n+1,n+2] \times \{0\}$ for every integer $n$. 
Consider another homotopy translation arc $\gamma$ for $f$ at the point $x_{1} = (0,0)$. The arc $\gamma$ is also a homotopy translation arc for the map $\Phi f$. Thus, up to replacing $f$ by $\Phi f$, we may assume that $f([n,n+1] \times \{0\})) = [n+1,n+2] \times \{0\}$ for every $n$. (The reader might be afraid that $\Phi f$ may have some fixed point, whereas $f$ was fixed point free, but we will not use this hypothesis anymore.) 

Now the map $f$ looks very much like the translation $T : (x,y) \to (x+1,y)$, and in a first reading the reader may assume that $f = T$.
We may assume that $\gamma$ is a geodesic (property 1 of the geodesics). If $\gamma$ is not homotopic to $[0, 1] \times {0}$, then $\gamma \not = [0,1] \times {0}$ and we will prove (as a contradiction) that gamma is not a homotopy translation arc
It is enough to prove that the geodesic  $f_{\sharp}(\gamma)$ homotopic to $f(\gamma)$ meets $\gamma$ at some point distinct from $(1,0)$.
For this we consider the two following families of curves:
$$
A = \{[n,n+1] \times \{0\}, n \in \bbZ\}, \ \ \ B = \{f(\gamma)\}. 
$$
These families satisfy the hypothesis of property 4 of the appendix, since $A$ and $\{\gamma\}$ do, and $f(A)=A$. Thus again there exists $\Phi \in \homeo_{0}(\bbR^2, \cO)$ such the image under $\Phi$ of all the curves in both families are geodesics, namely $\Phi([n,n+1] \times \{0\}) = [n,n+1] \times \{0\}$ for every $n$, and $\Phi f(\gamma) = f_{\sharp}(\gamma)$.
Since $\gamma$ is a geodesic distinct from and thus non homotopic to $[0,1] \times \{0\}$, it has to intersect the horizontal line $\bbR \times \{0\}$. Let $\gamma' \subset \gamma$ be the largest subarc containing $\gamma(0) = (0,0)$ and disjoint from this line except at its end-points. The other end-point of $\gamma'$ is on the horizontal line, say $(x,0)$. Since geodesics are in minimal position, this point does not belong to $[-1,1] \times \{0\}$. To fix ideas assume $x >1$. Since $\Phi f$ preserves the orientation and the horizontal line,
$\Phi f (\gamma')$ is an arc from $(1,0)$ to some point $(x',0)$ with $x'>x$ and otherwise disjoint from the line. From this we conclude that $\gamma' \cap \Phi f (\gamma') \neq \emptyset$, and thus $\gamma \cap f_{\sharp}(\gamma)$ contains a point distinct from $(1,0)$. This completes the proof.
\end{proof}

\begin{exercise}
Use the same techniques to prove that, for the Reeb map and $r=2$ (see the second picture), any homotopy translation arc for the point $x_{1}$ is homotopic to the horizontal translation arc drawn on the figure.
\end{exercise}

\subsection{Backward \emph{or} forward proper homotopy translation arcs}
If we consider more than three orbits we cannot in general find homotopy translation arcs that are both backward and forward proper.
However, Handel proved that there always exist homotopy translation arcs that are backward or forward proper. Even more, one can find for each of the $r$ orbits a backward proper homotopy translation arc, and a forward  proper homotopy translation arc, such that all the corresponding ``half homotopy streamlines'' are pairwise disjoint. Here we construct such a family in the special case of a homeomorphism satisfying the hypotheses of the fixed point theorem (see section~\ref{sec.orbit-diagrams} for the general statement).

We work in the same setting as in the previous section: $x_{1}, \dots , x_{r}$ are points having disjoint proper orbits for a homeomorphism $f$ of the plane. We use the same notations. The following property asks for the existence of a family of backward or forward proper homotopy translation arcs, whose associated ``homotopy half-streamlines'' are pairwise homotopically disjoint.

\bigskip

\noindent \textbf{Property $(H'_{1})$} 
\emph{There exists a positive integer $N$ and, for every $i=1 , \dots , r$, an arc
$\delta_{i} \in \cA$ joining $f^{-N-1} x_{i}$ to $f^{-N} x_{i}$, and an arc
$\gamma_{i} \in \cA$ joining $f^{N} x_{i}$ to $f^{N+1} x_{i}$, such that 
\begin{itemize}
\item the $\delta_{i}$'s are backward proper homotopy translation arcs,
\item the $\gamma_{i}$'s are forward proper homotopy translation arcs,
\item all the arcs in the family
$$
\{f^{-n}(\delta_{i}), f^{n}(\gamma_{i}) \ \   \mbox{ with }  i=1 , \dots , r, \ \ n \geq 0\}
$$
   are pairwise homotopically  disjoint.
\end{itemize}
}

\bigskip

\begin{prop}\label{prop.handel-matsumoto}
Let $f$ be a homeomorphism of the disk $\bbD^2$ with no fixed point in the interior.
Assume hypothesis $(H_{1})$ of Handel's fixed point theorem:  $x_{1}, \dots , x_{r}$ are points of the interior of the disk whose $\alpha$ and $\omega$-limit sets are distinct points $\alpha_{1}, \omega_{1}, \dots , \alpha_{r}, \omega_{r}$ on the boundary. Identify the interior of the disk with the plane. Then property $(H'_{1})$ holds for the restriction of $f$ to the interior of the disk.
\end{prop}

\begin{proof}
The idea is to define the $\delta_{i}$'s and the $\gamma_{i}$'s as geometrical translation arc for the euclidean metric on the disk, and to use the uniqueness of the homotopy class of translation arcs (proposition~\ref{prop.translation-arcs}) to prove homotopic disjointness.

\noindent \begin{minipage}[b]{0.6\linewidth}
The hypothesis allows to choose two collections $A_{1}, \dots, A_{r}$ and $W_{1}, \dots , W_{r}$ of pairwise disjoint neighborhoods of the points $\alpha_{1}, \omega_{1}, \dots , \alpha_{r}, \omega_{r}$, such that each $A_{i}, W_{i}$ is disjoint from the orbits of the $x_{j}$'s for $j \neq i$. More precisely, 
we construct $A_{i}$ (resp. $W_{i}$) as the intersection of a small disk centered at $\alpha_{i}$ (resp. $\omega_{i}$) with the unit disk $\bbD^2$, paying attention that the boundaries of $A_{i}$ and $W_{i}$ do not contain any point of $\cO$.
 Let $B(i,n)$ be the closed euclidean disk centered at $f^n(x_{i})$ and critical for $f$:
 \end{minipage}
\noindent \begin{minipage}[b]{0.4\linewidth}
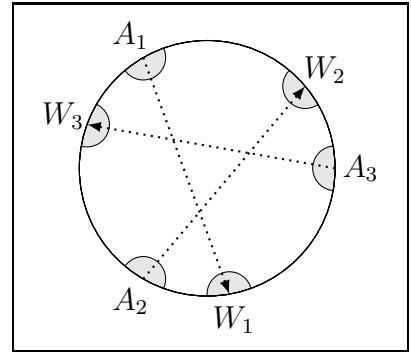
\begin{center}
\fbox{
\begin{tikzpicture}[scale=0.85]
\draw (0,0) circle (2);
\foreach \r in {3}
{
\foreach \i in {1,...,\r}
{\draw [thick, dotted, >=latex,->] (\i*360/\r:2) -- (\i*360/\r+480/\r:2) ;
\draw (\i*360/\r:2.4) node {$A_{\i}$};
\draw  (\i*360/\r+480/\r:2.4) node {$W_{\i}$};
\begin{scope}
\pgfsetfillopacity{0.1} 
\draw [clip] (0,0) circle (2);
\draw [fill] (\i*360/\r:2) circle (0.35); 
\draw [fill] (\i*360/\r+480/\r:2) circle (0.35);
\end{scope}}}
\end{tikzpicture} 
}
\end{center}
\end{minipage}
$$
f(B(i,n)) \cap B(i,n) \neq \emptyset \mbox{ but }
f(\inte B(i,n)) \cap \inte B(i,n) = \emptyset.
$$
For $n\geq0$ large enough, $B(i,n)$ in included in $W_{i}$, and so is its image under $f$. In particular, any geometric translation arc $\gamma(i,n)$ for $f^n(x_{i})$, as constructed in the previous section is included in $W_{i}$. Likewise, $B(i,-n-1)$ and its image are included in $A_{i}$, and so is any geometrical translation arc  $\delta(i,n)$ for $f^{-n-1}(x_{i})$.

From now on we work in the interior of the unit disk, identified with the plane.
Let $\gamma,\gamma'\in \cA$ be two simple arcs included in $A_{1}$. Assume that they are homotopic relative to the orbit $\cO(x_{1})$ of $x_{1}$. Then we observe that they are also homotopic relative to $\cO$.
Indeed, there exists a map $\Phi: \bbR^2 \to A_{1}$ which pointwise fixes $A_{1}$ and send the complement of $A_{1}$ to $\partial A_{1}$, and the composition of a homotopy avoiding $\cO(x_{1})$ with $\Phi$ gives a homotopy avoiding $\cO$. The same observation of course holds for all the $A_{i}$'s and $W_{i}$'s.

Now we choose $N > 0$ large enough so that for every $n \geq N$ and every $i$, the translation arc $\gamma(i,n)$ and its image are both included in $W_{i}$, and likewise 
the  $\delta(i,n), f^{-1}(\delta(i,n))$ are included in $A_{i}$. We set  $\gamma_{i} = \gamma(i,N)$ and $\delta_{i} = \delta(i,N)$ and claim that they suit our needs.

According to the proposition on homotopy classes of translation arcs, the arcs $\gamma(i,N+1)$ and $f(\gamma(i,N))$ are homotopic relative to $\cO(x_{i})$. Applying the above observation, we deduce that they are homotopic relative to $\cO$. By induction we see that the arc $f^n(\gamma_{i})$ is homotopic relative to $\cO$ to the arc $\gamma(i,N+n)$. Like wise the arc $f^{-n}(\delta_{i})$ is homotopic relative to $\cO$ to the arc $\delta(i,-N-n)$.
Thus property $(H'_{1})$ is satisfied.
\end{proof}

\bigskip
\bigskip
\bigskip
\section{Proof of the fixed point theorem}
In this section, we prove the intrinsic version of the theorem. The proof follows closely the exposition given by Matsumoto in~\cite{matsumoto2000arnold}.
Here the use of hyperbolic geodesics is crucial (see the appendix).

Again, consider an orientation preserving homeomorphism $f$ of the plane, with $r$ proper disjoint orbits $\cO(x_{1}), \dots, \cO(x_{r})$.
Assume property $(H'_{1})$. We want to translate in this setting hypothesis $(H_{2})$ concerning the cyclic order (see the statement of the fixed point theorem).

For a curve $\alpha \in \cA$, we will denote by $f_{\sharp}^n \alpha$ the unique geodesic in the homotopy class of $f^n(\alpha)$. Note that for any $p,q$ we have $f_{\sharp}^q f^p(\alpha) = f_{\sharp}^{p+q}\alpha$. We also define 
$$
S^-(\alpha) = \cup_{n \leq 0} f_{\sharp}^n\alpha, \ \ \ S^+(\alpha) = \cup_{n \geq 0} f_{\sharp}^n\alpha.
$$
Hypothesis $(H'_{1})$ amounts to saying that the curves
$$
S^-(\delta_{1}), S^+(\gamma_{1}), \dots , S^-(\delta_{r}), S^+(\gamma_{r}).
$$
are pairwise disjoint and homeomorphic to half-lines.
We also let 
$$
S^-= \cup_{i} S^-(\delta_{i}), \ \ \ S^+=\cup_{i} S^+(\gamma_{i}).
$$


As we did for flows (see the beginning of the section about classical Brouwer theory), we may consider the cyclic order at infinity. More precisely, in every neighborhood of infinity, that is, outside every compact subset of the plane, there exists a Jordan curve (a topological circle) $J$ meeting each of the $2r$ half-lines exactly once. The cyclic order induced by $J$ on the finite set $J \cap (S^- \cup S^+)$ does not depend on $J$, and thus we get a well defined cyclic order on our set of $2r$ half-lines. 
Denote by $\alpha_{i}$ the point $J \cap S^-(\delta_{i})$ and by $\omega_{i}$ the point $J \cap S^+(\gamma_{i})$.
We introduce hypothesis $(H'_{2})$ which says that the cyclic order on $J$ is the same as the order given on the circle boundary in hypothesis $(H_{2})$.

In view of proposition~\ref{prop.handel-matsumoto}, the fixed point theorem is a consequence of the following intrinsic statement.
\begin{theo*}[intrinsic version of Handel's fixed point theorem]
Let $f$ be an orientation preserving homeomorphism of the plane satisfying properties $(H'_{1})$ and $(H'_{2})$. Then $f$ has a fixed point.
\end{theo*}

The end of this section is devoted to the proof of this theorem.

\subsection{Action of $f$ on curves}
For the present we only assume that $f$ is an orientation preserving homeomorphism of the plane, with or without fixed points, satisfying hypothesis $(H'_{1})$.
We use the notations of section~\ref{sec.hta}.

We define the following subset
 $\cG$ of $\cA$. A curve $\alpha \in \cA$ is in $\cG$ if its end points belong to the set
  $$
  S^+ \cap \cO = \{f^{n}(x_{i}), n \geq N, i\in \{0, \dots , r\}\}
  $$
   and $\alpha$ is homotopically disjoint from every arc $f^n(\delta_{i}), n \leq 0$.
Note that $\cG$ is a union of homotopy classes in $\cA$ and that $f(\cG) \subset \cG$. 
Let $\cG_{0}\subset \cG$ be the set of elements of $\cG$ whose end points belong to the smaller set
$$
\{f^N(x_{1}), \dots , f^N(x_{r})\}
$$
and that are also homotopically disjoint from every arc $f^n(\gamma_{i}), n \geq 0, i=1, \dots, r$.

Obviously $f$ induces a natural map, denoted by $\underline{f}$, from $\underline{\cG}$ to itself, where $\underline{\cG}$ denotes the set of homotopy classes of elements of $\cG$.
We will also associate to $f$ a natural map from $\cG_{0}$ to itself. There is no obvious way to do this, mainly because the image of a curve $\alpha \in \cG_{0}$ may meet $S^+$. The idea is to cut $f(\alpha)$ into pieces that do not meet $S^+$ anymore, at least up to homotopy. Thus we will not obtain a genuine map but rather a  multi-valued map from $\underline{\cG}_{0}$ to itself.
More precisely, if all the curves involved are assumed to be pairwise in minimal position, then the process will be to take all the connected components of $f(\alpha)\setminus S^+$, and to extend them by the most direct way so that their end-points belong to $\{f^N(x_{1}), \dots f^N(x_{r})\}$. The result will be a ``set'' of curves in $\cG_{0}$ counted with multiplicities; for this we introduce the following notation.
For every set $E$, let $\oplus E$ denote the set ``finite subsets of $E$ with multiplicities''. More formally, an element of  $\oplus E$ is a  map $\phi$ from $E$ to $\bbN$
such that $\phi(e) = 0$ except for a finite number of $e\in E$. An element of $\oplus E$ may be denoted either by a formal sum $\phi = \alpha_{1} + \cdots + \alpha_{\ell}$, or (abusing notation) by a ``set'' $\{\alpha_{1}, \dots \alpha_{\ell}\}$ where the $\alpha_{i}$'s are not assumed to be distinct. The empty sum (or empty set) is denoted either by $0$ or $\emptyset$. We will write $\alpha \in \phi$ to denote $\phi(\alpha) \neq 0$.

Now consider some $\underline{\alpha} \in \underline{\cG}$, and define $\mathrm{cut}(\underline{\alpha}) \in \oplus\underline{\cG}_{0}$ as follows\footnote{Note that in \cite{handel1999fixed} the construction is slightly different. Our set $\cG_{0}$ is in one-to-one correspondance with the set which is denoted in~\cite{handel1999fixed} by $RH(W, \partial_{+}W)$, and Handel's map $f_{\sharp}(.) \cap W$ corresponds to our map $\mathrm{cut} \circ \underline{f}$.}.
In the case when $\alpha$ is isotopic to one of the geodesics $f_{\sharp}^n\gamma_{i}, n\geq0$ that make up $S^+$, we define  $\mathrm{cut}(\underline{\alpha})$ to be the empty set. In the opposite case, let $\alpha'$ be  the unique geodesic  homotopic to $\alpha$. According to the appendix, $\alpha'$ is in minimal position with all the geodesics $f_{\sharp}^n\gamma_{i}$. In particular  the set $B$ of connected components of $\alpha' \setminus S^+$ is finite. 
Let $\overline{\beta}$ be the closure of some element in $\beta \in B$, we consider $\overline{\beta}$ as an oriented simple curve parametrized by $[0,1]$, which connect some $S^+(\gamma_{i_{0}})$ to some $S^+(\gamma_{i_{1}})$.  We define a curve $\beta'$ by first following the half-line $S^+(\gamma_{i_{0}})$ from  $f^N(x_{i_{0}})$ to  $\overline{\beta}(0)$, then following $\overline{\beta}$, and finally following the half-line $S^+(\gamma_{i_{1}})$  from $\overline{\beta}(1)$ to  $f^N(x_{i_{1}})$. The curve $\beta'$ is then ``pushed off $S^+$'' to get a curve $\beta''$ which is disjoint from $S^+$ except at its end-points. 
The process from $\beta$ to $\beta''$ is described on the following picture.

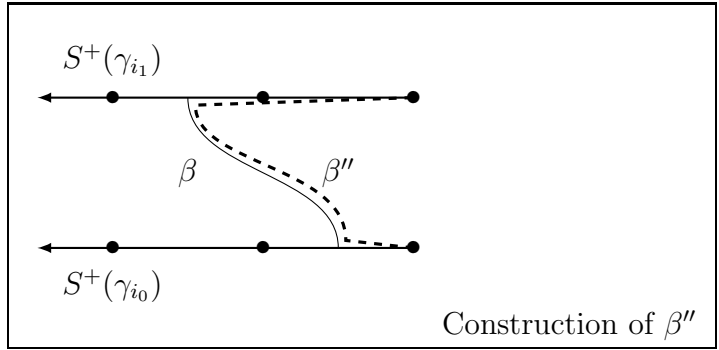
\begin{center}
\fbox{
\begin{tikzpicture}[scale=1]
\tikzstyle{fleche}=[>=latex,->, thick]
\draw [fleche] (5,0) -- (0,0) ; 
\draw [fleche] (5,2) -- (0,2) ;
\foreach \k in {1,3,5} 
	{\draw (\k,0) node {$\bullet$} ; \draw (\k,2) node {$\bullet$} ; }

\draw (4,0)  	..controls +(0,1)  and +(0,-1).. (2,2);

\draw [very thick,dashed]  (5,0) -- (4+0.1,0+0.1)  	
			..controls +(0,1)  and +(0,-0.8).. (2+0.1,2-0.1)
			-- (5,2);

\draw (2,1) node {$\beta$}; 
\draw (4,1) node {$\beta''$}; 
\draw (1,2.5) node {$S^+(\gamma_{i_{1}})$}; 
\draw (1,-0.5) node {$S^+(\gamma_{i_{0}})$};

\draw (0,-1) node {} ;
\draw (0,3) node {} ;
\end{tikzpicture} 

Construction of $\beta''$
}
\end{center}
It may happens that $\beta''$ is inessential (recall that this means it is a closed curve surrounding no point of $\cO$).
In this case we decide that $\beta''$ is the zero element in $\oplus \underline{\cG}_{0}$. In the opposite case $\beta''$ is an element of $\cG_{0}$.
Finally, we let
$$
\mathrm{cut}(\underline{\alpha})  = \sum_{\beta \in B} \underline{\beta}''.
$$

\begin{center}
\fbox{
\begin{tikzpicture}[scale=1]
\tikzstyle{fleche}=[>=latex,->, thick]
\draw [fleche] (5,0) -- (0,0) ; 
\draw [fleche] (5,2) -- (0,2) ;
\foreach \k in {1,3,5} 
	{\draw (\k,0) node {$\bullet$} ; \draw (\k,2) node {$\bullet$} ; }
\draw (4,-1) node {$\bullet$};
\draw (5,1) node {$\bullet$};

\draw 	(3,0) 	..controls +(0,-1)  and +(-1,0)..
		(4,-1.5)	..controls +(0.3,0)  and +(0,-0.3)..
		(4.5,-1)	..controls +(0,0.5)  and +(0,-1)..
		(4,0)  	..controls +(0,1)  and +(0,-1)..
		(2,2)		..controls +(0,1)  and +(0,1)..
		(5,2);

\draw [very thick,dashed]  (5,0) ..controls +(-0.5,0.5)  and +(-0.5,-0.5)..   (5,2);
\draw [very thick,dashed]  (5,0) ..controls +(-1.2,-0.5)  and +(-0.3,0.3)..
					(3.8,-1.2) ..controls +(0.3,-0.3)  and +(-0.5,-1.2)..   (5,0);

\draw (2.5,-1) node {$\beta_{1}$}; 
\draw (2.4,1) node {$\beta_{2}$}; 
\draw (2.3,3) node {$\beta_{3}$}; 

\draw (5.2,-0.6) node {$\underline{\beta''_{1}}$}; 
\draw (4.3,1) node {$\underline{\beta''_{2}}$}; 
\draw (5,3) node {$(\underline{\beta''_{3}} = 0)$}; 

\draw (0.2,2.5) node {$S^+$}; 
\end{tikzpicture} 
}

An example : $\mathrm{cut}(\alpha) = \underline{\beta''_{1}} + \underline{\beta''_{2}}$
\end{center}

For future use we make the following remark. Imagine that at the beginning of the above construction we replace the geodesic $\alpha'$ by any curve wich is homotopic to $\alpha$ and in minimal position with all the $f_{\sharp}^n\gamma_{i}$ with $n \geq0$. Due to properties of the geodesics, such a curve is the image of  $\alpha'$ under some $\Phi \in \mathrm{Homeo}_{0}(\bbR^2, \cO)$ which leaves every  geodesic $f_{\sharp}^n\gamma_{i}$ globally invariant (apply property 4 of the appendix on geodesics). Then if we apply the construction with $\Phi(\alpha')$ instead of $\alpha'$, we will get the curves $\Phi(\beta'')$ instead of $\beta''$. Since each $\Phi(\beta'')$ is homotopic to $\beta''$ relative to $\cO$, we see that this will not change the definition of $\mathrm{cut}(\underline{\alpha})$.

\newpage
\begin{exercise}
Prove that this definition does not depend on the choice of the hyperbolic structure: if $\cH_{0}, \cH_{1}$ are two families of curves satisfying the axioms of geodesics listed in the appendix, then the maps $\mathrm{cut}_{0}$ and $\mathrm{cut}_{1}$ defined using respectively  $\cH_{0}, \cH_{1}$ coincide. Hint: use again property 4 in the list of axioms.
\end{exercise}

Thus we have well-defined maps $\underline{f}: \underline{\cG} \to \underline{\cG}$ and  $\mathrm{cut} : \underline{\cG} \to \oplus \underline{\cG}_{0}$.
We still denote by $\underline{f}:\oplus \underline{\cG} \to \oplus \underline{\cG}$ and  $\mathrm{cut} : \oplus \underline{\cG} \to \oplus \underline{\cG}_{0}$ the natural extensions.


\bigskip

\begin{lemm}~

\begin{enumerate}
\item $\mathrm{cut}\circ \mathrm{cut} = \mathrm{cut}$.
\item Let  $\alpha_{1}, \alpha_{2} \in {\cG}$ be homotopically disjoint. Then every $\beta_{1} \in \mathrm{cut}(\underline{\alpha_{1}}),   \beta_{2} \in \mathrm{cut}(\underline{\alpha_{2}})$ are homotopically disjoint.
\item The map $\underline{f}$ depends only on the homotopy class of $f$ relative to $\cO$: if $\Phi$ is any element in $\mathrm{Homeo}_{0}(\bbR^2, \cO)$ then $\underline{\Phi f} = \underline{f}$.
\item For every integer $n\geq 0$, the equality 
$$
(\mathrm{cut} \circ \underline{f})^n = \mathrm{cut} \circ \underline{f}^n
$$
holds on $\oplus \underline{\cG}$.
\end{enumerate}
\end{lemm}

\begin{proof}
The first point simply expresses the fact that the restriction of $\mathrm{cut}$ to $\underline{\cG_{0}}$ is the identity. For the second point, consider two connected components $\beta_{1},\beta_{2}$ coming respectively from $\alpha_{1}, \alpha_{2}$ as in the definition of the map $\mathrm{cut}$. Since geodesics have minimal intersection, these two components are disjoint. Then it is easy to choose the curves $\beta''_{1}, \beta''_{2}$ so that they are disjoint (except maybe for their end-points, as usual). This proves the homotopic disjointness.
The third point is obvious.

Let us turn to the last point. By writing $(\mathrm{cut} \circ \underline{f})^n = (\mathrm{cut} \circ  \underline{f}) \circ (\mathrm{cut} \circ  \underline{f})^{n-1}$ and using induction we see that
 it suffices to show that $\mathrm{cut}  \underline{f} \mathrm{cut} =  \mathrm{cut}  \underline{f}$ on $\underline{\cG}$. 
According to the previous point, we may modify $f$ before doing the computation, as long as we do not change the homotopy class relative to $\cO$. We claim that this will allow us to assume the following additional property:
\emph{for every $n \geq -1$, $f$ sends the geodesic $f_{\sharp}^n\gamma_{i}$ to the geodesic $f_{\sharp}^{n+1}\gamma_{i}$}.
Indeed, consider the family $\cF = \{f_{\sharp}^n\gamma_{i}, i = 1, \dots, r, n \geq -1 \}$. 
According to hypothesis $(H'_{1})$ each $\gamma_{i}$ is a forward proper homotopy translation arc, thus this family is locally finite (this makes use of property 5 of the appendix). Furthermore all the arcs in the family $\{f^{n}(\gamma_{i}) \ \   i=1 , \dots , r, \ \ n \geq 0\}$   are pairwise homotopically  disjoint. Since $\cF$ is obtained from this family by first applying $f^{-1}$ and then replacing each arc by the geodesic homotopic to it, we see that the elements of $\cF$ are pairwise disjoint (this makes use of property 2 of the appendix). The image family $f(\cF)$ is again locally finite with pairwise disjoint elements; according to property 4 of the appendix, there exists some $\Phi \in \mathrm{Homeo}_{0}(\bbR^2, \cO)$ such that all the $\Phi(f( f_{\sharp}^n\gamma_{i})), n\geq-1$ are geodesics. Uniqueness of geodesics implies that 
$\Phi(f(f_{\sharp}^n\gamma_{i})) = f_{\sharp}^{n+1}\gamma_{i}$ for every $i$ and every $n \geq -1$. We may replace $f$ by $\Phi f$ to get the above additional property. Note that this does not affect the map $f_{\sharp}$.

Now let $\alpha \in \cG$, we want to check that $\mathrm{cut}  \underline{f} \mathrm{cut} (\underline{\alpha})=  \mathrm{cut}  \underline{f} (\underline{\alpha})$.
For this we may assume that $\alpha$ is a geodesic. In particular $\alpha$ is in minimal position with every element of the above family $\cF$, and thus $f(\alpha)$ is in minimal position with every element of the image family; thanks to the preliminary modification on $f$, this family is exactly the family of geodesics that make up $S^+$.
Thus, using the notations of the definition of the map $\mathrm{cut}$ and according to the remark following this definition, we have
$$
\mathrm{cut} \underline{f(\alpha)} = \sum_{\beta_{1} \in B(f(\alpha))} \underline{\beta}_{1}'' \ \ \ (\star)
$$
where $B(c)$, for any curve $c$, stands for the set of connected components of $c \setminus S^+$.

On the other hand let us determine what is $\mathrm{cut}  \underline{f} \mathrm{cut} \underline{\alpha}$. 
Since $\alpha$ is a geodesics we have
$$
\mathrm{cut} \underline{\alpha} = \sum_{\beta_{0} \in B(\alpha)} \underline{\beta}_{0}''.
$$
Applying $\mathrm{cut}  \underline{f}$ to this sum yields
$$
\mathrm{cut}  \underline{f} \mathrm{cut} \underline{\alpha} = \sum_{\beta_{0} \in B(\alpha)} \mathrm{cut}  \underline{f}(\underline{\beta}_{0}'').
$$
Because of the preliminary modification of $f$ we have $f^{-1}(S^+) = S^+ \cup \cup_{i=1, \dots,r} f_{\sharp}^{-1}\gamma_{i}$, and each geodesic $f_{\sharp}^{-1}\gamma_{i}$ meets $S^+$ only at one end-point. 
Each $\beta_{0}''$ is made of an arc included in $\alpha$ and two arcs close to $S^+$, which may be chosen to be disjoint from the $f_{\sharp}^{-1}\gamma_{i}$'s. Since $\alpha$ is in minimal position with all the $f_{\sharp}^{n}\gamma_{i}$'s with $n \geq -1$, we deduce that the arcs $\beta_{0}''$ are also in minimal position with all these geodesics. Then the arcs $f(\beta_{0}'')$ is in minimal position with all the curves that make up $S^+$. This allows us to write
$$
\mathrm{cut} \underline{f}\underline{\beta}_{0}'' = \sum_{\tilde \beta_{1} \in B(f(\beta_{0}''))} \underline{\tilde \beta}_{1}'' \ \ \ (\star \star).
$$
Now since $f(S^+) \subset S^+$, every connected component $\beta_{1}$ of $f(\alpha) \setminus S^+$ is included in a (unique) connected component$f(\beta_{0})$ of $f(\alpha) \setminus f(S^+)$. 
Thus the sum $(\star)$ writes as the sum over $\beta_{0} \in B(\alpha)$ of  the terms
$$
\sum_{\beta_{1} \in B(f(\beta_{0}))} \underline{\beta}_{1}'' \ \ \ (\star\star\star).
$$
To get the equality it remains to compare the sums $(\star\star)$ and $(\star\star\star)$. For this we have to compare their sets of indices.
Remember that $\beta_{0}''$ is obtained from $\beta_{0}$ by adding at both end-points some arcs included in $S^+$, and then making a small perturbation to push the resulting curve off $S^+$.
Thus there is a bijection $\beta_{1} \mapsto \tilde \beta_{1}$ between both  sets of indices, with $\tilde \beta_{1} = \beta_{1}$ except for the two extreme components of $f(\beta_{0}) \setminus S^+$ (which may coincide in the case when there is only one component) for which $\tilde \beta_{1}$ is obtained from $\beta_{1}$ by adding at one (or both)  end-point an arc included in $f(S^+)$ and making a small perturbation. In any case, the homotopy classes of $\beta_{1}''$ and $\tilde \beta_{1}''$ are easily seen to coincide. This completes the proof of the equality.
\end{proof}


Following Handel, we denote by $-\alpha$ the curve $\alpha$ with reverse orientation (note however that the formal sum $-\underline{\alpha}+\underline{\alpha}$ in $\oplus \underline{\cG}$ is not equal to zero!).
The interest of the map $\mathrm{cut} \circ  \underline{f}$ appears in the following crucial statement.
\begin{prop}\label{prop.fixed-point}
If there is some $\alpha \in \cG_{0}$ and some positive $n$  such that 
$$
-\underline{\alpha} \in \mathrm{cut} \underline{f}^n(\underline{\alpha})
$$
 then $f$ has a fixed point. 
\end{prop}
\begin{center}
\fbox{
\begin{tikzpicture}[scale=1]
\tikzstyle{fleche}=[>=latex,->, thick]
\draw [fleche] (5,0) -- (0,0) ; 
\draw [fleche] (5,2) -- (0,2) ;
\foreach \k in {1,3,5} 
	{\draw (\k,0) node {$\bullet$} ; \draw (\k,2) node {$\bullet$} ; }

\draw (1.8,2.8) node {$f^n(\alpha)$}; 
\draw (2.5,0.8) node {$\beta_{2}$}; 
\draw 	(1,2) 	..controls +(0.3,-2)  
		and +(0,-1.5)..
		(4,0)  	..controls +(0,0.7) and +(-20:0.5)..
		 (3,1)	..controls +(160:0.5) and +(0,-0.7)..
		(2,2)		..controls +(0,1)  and +(0,3)..
		(6,1)		..controls +(0,-2)  and +(2,0)..
		(4,-1)	..controls +(-2,0)  and +(0,-1)..
		(1,0);
\draw [->,very thick]  (2,2)  	..controls +(0,-0.7) and +(160:0.5)..   (3,1);
\draw [very thick]   (3,1)	..controls +(-20:0.5) and +(0,0.7).. (4,0);

\draw (5.3,1) node {$\alpha$}; 
\draw   [->,very thick]   (5,0) --    (5,1);
\draw   [very thick]   (5,1) --    (5,2);





\draw (0.2,2.5) node {$S^+$}; 
\end{tikzpicture} 
}

Hypothesis of proposition~\ref{prop.fixed-point}: $-\underline{\alpha} \in \mathrm{cut} \underline{f}^n\underline{\alpha}$
\end{center}

\begin{proof}
This is the only place where we will use the existence of a nice circle boundary for the universal cover. Let $\pi : \bbH^2 \to \bbR^2 \setminus \cO$ be the universal cover given by the theorem in the appendix. We know that every lift $\tilde h : \bbH^2 \to \bbH^2$ of a homeomorphism $h$ of $\bbR^2 \setminus \cO$ extends to the circle boundary $\partial  \bbH^2 \simeq \bbS^1$. Under the hypotheses of the proposition, we claim that \emph{there exists some lift $\tilde h$ of $h = f^n$ whose extension has no fixed point on the boundary.} The claim easily implies the proposition by the following argument. We apply Brouwer fixed point theorem to the homeomorphism $\tilde h$ of the closed  two-disk $\bbH^2 \cup \partial  \bbH^2$. Since $\tilde h$ has no fixed point on the boundary, it must have a fixed point in $\bbH^2$. Thus $h$ has a fixed point. That is, $f$ has a periodic point. Finally the Brouwer plane translation theorem provides a fixed point for $f$ (see section~\ref{sec.brouwer-theory}.\ref{sub.Brouwer}).

Let us prove the claim. We may assume that $\alpha$ is a geodesic. We also note that we may modify $h$ within its homotopy class relative to $\cO$, since this does not affect the restriction of lifts of $h$ on the boundary of $\bbH^2$. Thus, as in the proof of the previous lemma, we may assume that $h(S^+) \subset S^+$, and also that $\alpha'=h(\alpha)$ is a geodesic (this makes use of property 4 in the appendix). 
The end-points of $\alpha$ are some points $f^N(x_{i_{0}}), f^N(x_{i_{1}})$ of $\cO$.
Remember that every lift $\tilde \gamma$ of a curve $\gamma$ in $\cA$ has end-points $\tilde \gamma(0), \tilde \gamma(1)$ on $\partial \bbH^2$, and furthermore that these end-points depend only on the homotopy class of $\gamma$ relative to $\cO$. 
Consider an infinite curve $c_{0}$ surrounding $S^+(\gamma_{i_{0}})$  as on the picture below (on the left). 

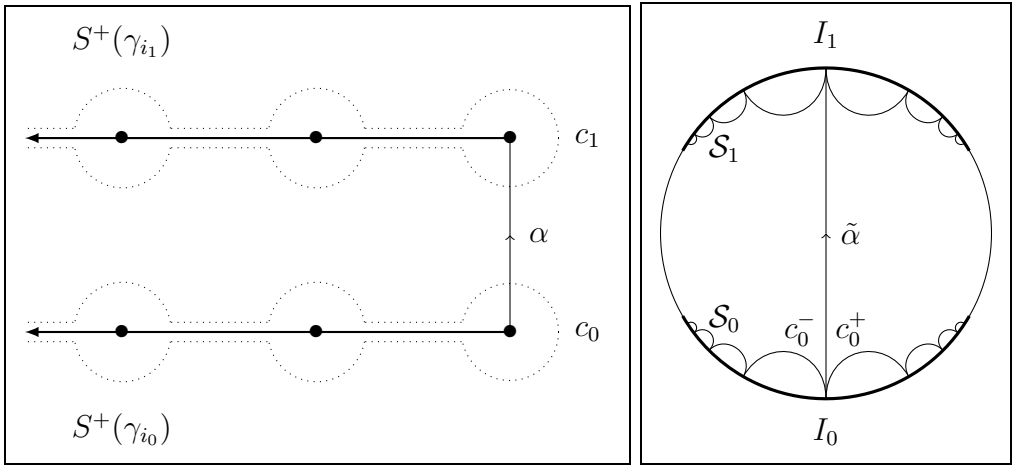
\begin{center}
\fbox{
\begin{tikzpicture}[scale=1.29]
\tikzstyle{fleche}=[>=latex,->, thick]
\draw [fleche] (5,0) -- (0,0) ; 
\draw [fleche] (5,2) -- (0,2) ;
\foreach \k in {1,3,5} 
	{\draw (\k,0) node {$\bullet$} ; \draw (\k,2) node {$\bullet$} ; }

\def\eps{0.5}
\def\ang{10}
\def\epsp{0.1}
\foreach \k in {0,2} 
{
\draw [dotted] (5+\eps,\k) arc (0:180-\ang:\eps);
\draw  [dotted] (5-\eps,\k+\epsp) -- (3+\eps,\k+\epsp);
\draw  [dotted] (3+\eps,\k+\epsp) arc (\ang:180-\ang:\eps);
\draw  [dotted] (3-\eps,\k+\epsp) -- (1+\eps,\k+\epsp);
\draw  [dotted] (1+\eps,\k+\epsp) arc (\ang:180-\ang:\eps);
\draw  [dotted] (1-\eps,\k+\epsp) -- (0,\k+\epsp);

\draw  [dotted] (5+\eps,\k) arc (0:-180+\ang:\eps);
\draw [dotted]  (5-\eps,\k-\epsp) -- (3+\eps,\k-\epsp);
\draw [dotted]  (3+\eps,\k-\epsp) arc (-\ang:-180+\ang:\eps);
\draw [dotted]  (3-\eps,\k-\epsp) -- (1+\eps,\k-\epsp);
\draw [dotted]  (1+\eps,\k-\epsp) arc (-\ang:-180+\ang:\eps);
\draw [dotted]  (1-\eps,\k-\epsp) -- (0,\k-\epsp);
}

\draw (5.8,2) node {$c_{1}$}; 
\draw (5.8,0) node {$c_{0}$}; 
\draw (1,3) node {$S^+(\gamma_{i_{1}})$}; 
\draw (1,-1) node {$S^+(\gamma_{i_{0}})$}; 

\draw (5.3,1) node {$\alpha$}; 
\draw   [->]   (5,0) --    (5,1);
\draw    (5,1) --    (5,2);

\end{tikzpicture} 
}
\fbox{
\begin{tikzpicture}[scale=1.1]
\draw (0,0) circle (2);
\foreach \k in {1,-1}
{
\foreach \l in {0,180}
{
\draw (\l+\k*90:2) ..controls +(180+\l+\k*90:0.6)  and +(180+\l+\k*60:0.6)..   (\l+\k*60:2);
\draw (\l+\k*60:2) ..controls +(180+\l+\k*60:0.3)  and +(180+\l+\k*45:0.3)..   (\l+\k*45:2);
\draw (\l+\k*45:2) ..controls +(180+\l+\k*45:0.2)  and +(180+\l+\k*37:0.2)..   (\l+\k*37:2);
\draw (\l+\k*37:2) ..controls +(180+\l+\k*37:0.1)  and +(180+\l+\k*33:0.1)..   (\l+\k*33:2);
}}
\draw (-75:1.2) node {$c^+_{0}$}; 
\draw (-105:1.2) node {$c^-_{0}$}; 
\draw (-140:1.6) node {$\cS_{0}$}; 
\draw (140:1.6) node {$\cS_{1}$}; 

\draw (90:2.4) node {$I_{1}$};
\draw (-90:2.4) node {$I_{0}$};

\draw [very thick] (30:2) arc (30:150:2);
\draw [very thick] (-30:2) arc (-30:-150:2);

\draw (0.3,0) node {$\tilde \alpha$};
\draw [->] (0,-2) -- (0,0);
\draw (0,0) -- (0,2);

\end{tikzpicture} 
}
\end{center}

Let $\tilde c_{0}$ be any lift of $c_{0}$.
For each $n \geq 0$ there exist exactly two lifts $\tilde c_{\pm}^n$ of $f_{\sharp}^n(\gamma_{i_{0}})$ meeting $\tilde c_{0}$, and these lifts form a sequence
$$
\cS_{0} =  \dots, c_{-}^{2}, c_{-}^{1} , c_{-}^{0}, c_{+}^{0}, c_{+}^{1}, c_{+}^{2},\dots 
$$
in which the terminal end-point (on $\partial \bbH^2$) of some term coincides with the initial end-point of the next term (see the above picture, on the right).

Let $\tilde \alpha$ be the lift of $\alpha$ meeting $\tilde c_{0}$, let $\tilde c_{1}$ be the lift of $c_{1}$ meeting $\tilde \alpha$ where $c_{1}$ is an infinite curve surrounding $S^+(\gamma_{i_{1}})$. Let $\cS_{1}$ be the sequence of lifts of the geodesics $f_{\sharp}^n(\gamma_{i_{1}})$ defined analogously to $\cS_{0}$. Since the curves $c_{0}$ and $c_{1}$ are disjoint we see that the situation is an on the picture. In particular, if $I_{0}$ is the minimal open interval in the boundary of $\bbH^2$ that contains all the end-points of the elements of the sequence $\cS_{0}$, and $I_{1}$ is defined similarly, then $I_{0}$ and $I_{1}$ are disjoint. Since $h(S^+(\gamma_{i_{0}}))$ is included in $S^+(\gamma_{i_{0}})$ and thus disjoint from $S^+(\gamma_{i_{1}})$, we get that $c_{1}$ may be chosen to be disjoint from  $h(c_{0})$, and thus we see that no lift of $h(c_{0})$ intersects $c_{1}$. We deduce that for every lift $\tilde h$ of $h$ the interval $\tilde h(I_{0})$ is either disjoint from, included in or containing $I_{1}$. The same holds when we exchange $I_{0}$ and $I_{1}$.

Now the hypothesis $-\underline{\alpha} \in \mathrm{cut} \underline{f}^n\underline{\alpha}$ says that there exists some component $\beta$ of $f^n(\alpha) \setminus S^+$ such that $\beta''$ is homotopic to $\alpha$ with the reverse orientation (we have again endorsed the notations in the construction of the map $\mathrm{cut}$).
Thus $\beta''$ has a lift $\tilde \beta''$ such that $\tilde \beta''(0) = \tilde \alpha(1)$ and $\tilde \beta''(1) = \tilde \alpha(0)$.
If $\tilde \alpha'$ denotes the lift of $f^n(\alpha)$ that meets $\tilde \beta''$, then $\tilde \alpha'$ is an oriented  geodesic containing a subarc from a point in a geodesic $g_{1}$ of the sequence $\cS_{1}$ to a point in a geodesic $g_{0}$ of the sequence $\cS_{0}$. Let $\tilde h$ be the lift of $h$ that sends $\tilde \alpha$ to $\tilde \alpha'$. The curve $\tilde \alpha'$ may not meet $g_{1}$ more than once since these curves are geodesics and thus in minimal position (property 2 of the geodesics), and thus we see that the initial end-point of $\tilde \alpha'$ is included in $I_{1}$. Likewise the terminal end-point of $\tilde \alpha'$ is included in $I_{0}$. Since these points belongs respectively to $\tilde h(I_{0})$ and $\tilde h (I_{1})$, combining  this with previous observations we get that
$$
\tilde h(I_{0}) \subset I_{1}, \ \ \  \tilde h(I_{1}) \subset I_{0}.
$$
Thus $\tilde h$ has no fixed point in $I_{0}$ neither in $I_{1}$. Let $J_{0}, J_{1}$ be the complementary intervals of $I_{0} \cup I_{1}$ in the circle.
Since $\tilde h$ preserves orientation on the boundary, it must sends $J_{0}$ into $I_{0} \cup J_{1} \cup I_{1}$. Likewise $\tilde h (J_{1})$ is disjoint from $J_{1}$. Finally $\tilde h$ has no fixed point on the boundary, and the proof is complete.
\end{proof}

\subsection{Construction of a fitted family $T$}\label{ss.construction-fitted}
Under the assumptions $(H'_{1})$ and $(H'_{2})$ of the fixed point theorem, we look for some simple curve $\alpha$ and some positive $n$ such that $-\underline{\alpha} \in \mathrm{cut} \underline{f}^n\underline{\alpha}$. To this aim we will ``iterate and cut'' the curves $\delta_{i}$.
It is easy to see that for every $n \geq 2N+1$ the curve $f^n(\delta_{i})$ belongs to $\cG$, so that $\mathrm{cut} \underline{f}^n(\delta_{i})$ is well defined. Let
$$
T = \left\{\underline{\alpha} \in \mathrm{cut} \circ \underline{f}^n(\delta_{i}), i=1, \dots, r, n \geq 2N+1\right\}
$$
considered as a set \emph{without} multiplicity (otherwise some elements could have infinite multiplicity). This is a subset of $\underline{\cG}_{0}$.

\begin{lemm}[Existence of a fitted family]~
\label{lemm.existence-fitted}

\begin{enumerate}
\item (disjointness)  Every $\underline{\alpha}_{1} \neq \underline{\alpha}_{2} \in T$ are homotopically disjoint;
\footnote{Note that this does not exclude the possibility that $\underline{\alpha}_{1} = -\underline{\alpha}_{2}$.}
\item (finiteness) $T$ is a finite set;
\item (dynamical invariance) for every $\underline{\alpha}_{1} \in T$, every $\underline{\alpha}_{2} \in \mathrm{cut} \underline{f}(\underline{\alpha}_{1})$ belongs to $T$;
\item (non triviality) under hypothesis $(H'_{2})$, the family $T$ is non-empty, and it contains an element $\underline{\alpha}$ with distinct end-points.
\end{enumerate}
\end{lemm}
A set satisfying items 1,2,3 is called a \emph{fitted family}.

\subsubsection*{An example}
\noindent \begin{minipage}[b]{0.6\linewidth}
We describe an example satisfying the hypotheses of the fixed point theorem (with a fixed point!).
As for our previous example, it will be constructed as a perturbation of a flow.
First consider a map $f$ which is the time one map of a  flow of the closed disk as on the first picture, and six fixed points $\alpha_{1}, \dots, \omega_{3}$ on the boundary.
On the second picture we indicate how to modify $f$ into a homeomorphism $f'=\phi \circ f$ so that, after the modification, $\alpha_{i},\omega_{i}$ are the $\alpha$ and $\omega$ limit point of  a point $x_{i}$.
\end{minipage}
\noindent \begin{minipage}[b]{0.4\linewidth}
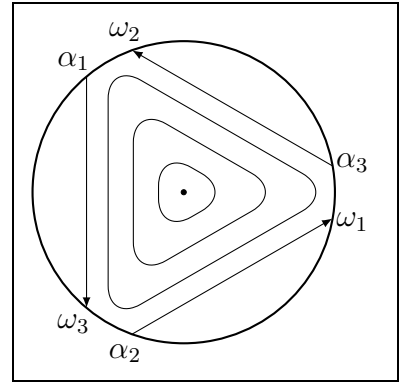
\begin{center}
\fbox{
\begin{tikzpicture}[scale=0.67]
\tikzstyle{fleche}=[>=latex,->]
\draw [thick] (0,0) circle (3cm) ;
\draw (0,0) [fill] circle (0.05) ;

\foreach \i in {1,2,3}  
{\draw [fleche] (\i*120+10:3) -- (\i*120+110:3)  ;
\draw (\i*120+10:3.4) node {$\alpha_{\i}$};
\draw  (\i*120+230:3.4) node {$\omega_{\i}$};
}

\clip  (0,0) circle (3cm) ;
\foreach \k in {1,2,3}
	{\draw [rounded corners=0.5cm] (0:\k cm) -- (120:\k cm) -- (240:\k cm) -- cycle ;}

\end{tikzpicture} 
}
\end{center}
\end{minipage}

\begin{figure}[p]
\begin{center}
\fbox{
\begin{tikzpicture}[scale=1.5]
\tikzstyle{fleche}=[>=latex,->]
\draw [thick] (0,0) circle (3cm) ;

\foreach \i in {1,2,3}  
{\draw [fleche] (\i*120+10:3) -- (\i*120+110:3)  ;
\draw (\i*120+10:3.4) node {$\alpha_{\i}$};
\draw  (\i*120+230:3.4) node {$\omega_{\i}$};
\pgfsetfillopacity{0.1}
\filldraw [rotate=\i*120-120]   (43+120:1.6) ellipse (0.6 and 0.1);
\filldraw [rotate=\i*120-120]   (-43-120:1.6) ellipse (0.6 and 0.1);
\pgfsetfillopacity{1}

\draw [fleche,rotate=\i*120] (-1.7,0.47) -- (-1.3,0.47) ;
\filldraw [rotate=\i*120] (-1.15,0.47) circle (0.04) ;
\filldraw [rotate=\i*120] (-1.93,0.47) circle (0.04) ;

\draw [fleche,rotate=\i*120] (-1.3,-0.47) -- (-1.7,-0.47) ;
\filldraw [rotate=\i*120] (-1.15,-0.47) circle (0.04) ;
\filldraw [rotate=\i*120] (-1.93,-0.47) circle (0.04) ;

\filldraw [rotate=\i*120] (0:2.13) circle (0.04) ;
\draw [rotate=\i*120+120] (0:2.4) node {$x_{\i}$} ;
}

\draw [fleche] (-119:2) .. controls +(50:0.8) and +(10:-0.8) ..(0.8,-0.8) ;
\draw [>=latex,<-] (-121:2) .. controls +(70:0.8) and +(110:-0.8) ..(-1.05,0.35) ;

\def \k {2.3}
\draw [rounded corners=0.5cm] (0:\k) -- (120:\k) -- (240:\k) -- cycle;

\end{tikzpicture} 
}
\fbox{
\begin{tikzpicture}[scale=1.5]
\tikzstyle{fleche}=[>=latex,->]
\draw [thick] (0,0) circle (3cm) ;

\foreach \i in {1,2,3}  
{\draw  (\i*120+10:3) -- (\i*120+110:3)  ;
\draw (\i*120+10:3.4) node {$\alpha_{\i}$};
\draw  (\i*120+230:3.4) node {$\omega_{\i}$};
\draw [dotted,rotate=\i*120-120]   (43+120:1.6) ellipse (0.6 and 0.1);
\draw [dotted,rotate=\i*120-120]   (-43-120:1.6) ellipse (0.6 and 0.1);
\filldraw [rotate=\i*120] (-1.15,0.47) circle (0.04) ;
\draw [rotate=\i*120] (-1.93,0.47) circle (0.04) ;

\draw [rotate=\i*120] (-1.15,-0.47) circle (0.04) ;
\filldraw [rotate=\i*120] (-1.93,-0.47) circle (0.04) ;

\filldraw [rotate=\i*120] (0:2.13) circle (0.04) ;
\draw [rotate=\i*120+120] (0:2.4) node {$x_{\i}$} ;

\draw [rotate=\i*120,ultra thick] (1*120+10:3) -- (-1.93,0.67)  ;
\filldraw [rotate=\i*120]  (-1.93,0.67) circle (0.04) ;
\draw [rotate=\i*120,ultra thick] (-1.93,0.67) .. controls +(0,-0.3) and +(0,0.2) .. (-1.15,0.47);
\draw [rotate=\i*120,ultra thick,fleche] (-1.93,-0.47)  -- (1*120+110:3) ; 
}

\def \k {2.3}
\draw [rounded corners=0.5cm] (0:\k) -- (120:\k) -- (240:\k) -- cycle;

\end{tikzpicture} 
}\\
The sets $S^-(\delta_{i})$ and $S^+(\gamma_{i})$
\end{center}
\end{figure}

The points $x_{i}$ will met hypothesis $(H_{1}),(H_{2})$ of the fixed point theorem. And, of course, the restriction to the open disk will met the corresponding hypotheses $(H'_{1}),(H'_{2})$. The map $\phi$ is the commutative product of six maps supported on pairwise disjoint topological disks which are free for $f$. Here, in the notations on hypothesis $(H'_{1})$ we may choose $N=1$, and the properly embedded half-lines $S^-(\delta_{i}), S^+(\gamma_{i})$ are indicated in thick lines on the third picture.

\newpage
\begin{exercise}
Describe the fitted family $T$ on this example. Describe the dynamics induced by $f'$ on this family, by drawing a graph $\Gamma_{T}$ whose vertices are the elements of the family, and one arrow from ${t}$ to each element of $\mathrm{cut} \circ \underline{f'} ({t})$. This graph will play an important part in the proof of the theorem. Hint: there are twelve elements, and for each element ${t}\in T$, the arc $-{t}$ with opposite orientation also occurs in $T$. Solution: see the second appendix.
\end{exercise}

\begin{proof}[Proof of lemma~\ref{lemm.existence-fitted}]
Due to hypothesis $(H'_{1})$ the curves in the set 
$$
\{f^n(\delta_{i}), n \geq 0, i=1, \dots, r\}
$$
 are pairwise homotopically disjoint.  Since the map $\mathrm{cut}$ preserves homotopic disjointness (previous lemma), we get the first point.

According to the first point, for the second point it is enough to bound the number of disjoint non-homotopic simple curves in $\cG_{0}$.
The situation amounts to the following problem. Consider a closed disk with $r$ marked points on the boundary, and $\ell = (2N-1)r$ punctures in the interior. Consider a family of simple curves avoiding the punctures, with each end-point equal to one marked point on the boundary, pairwise disjoint and non-homotopic. Let $N(r,\ell)$ be the maximum number of curves in such a family. An immediate induction based on the following estimate shows that $N(r,\ell) < +\infty$ for every $r,\ell$.
\begin{exercise}
Prove that  $N(r,0) \leq r^2$ and $N(r,\ell) \leq 2r^2 + 2N(r+1,\ell-1)$.
\end{exercise}

The third point is a consequence of the equality $(\mathrm{cut}  \underline{f})^n = \mathrm{cut} (\underline{f}^n)$.

For the last point we begin by the following observation. Assume some curve $\alpha \in \cG$ is not homotopically disjoint from some $f_{\sharp}^n(\gamma_{j})$ with $n \geq 2N+1$, and assume that $S^+(\gamma_{j})$ do not contain both end-points of $\alpha$. Then the sum $\mathrm{cut}\underline{\alpha}$ contains some element with distinct end-points.
Thus for the last point it suffices to prove that for some $n\geq2N+1$, and some $i \neq j$, the curve $f^n(\delta_{i})$ is not homotopically disjoint from at least one of the curves that make up  $S^+(\gamma_{j})$.
We will work with geodesics, and use repeatedly that two geodesics are disjoint 
as soon as their homotopy classes are, and that the curves $S^+(\delta_{i})$ are positively invariant up to homotopy.
Assume that the geodesic $f_{\sharp}^n\delta_{1}$ is disjoint from $S^+(\gamma_{r})$ for every $n \geq 2N+1$ (otherwise the point is proved).
Iterating negatively, we get that $f^{-\ell}(f_{\sharp}^{2N+1}\delta_{1})$ is disjoint from $f^{-\ell}(S^+(\gamma_{r}))$. Thus $f_{\sharp}^n\delta_{1}$ is disjoint from $S^+(\gamma_{r})$ for every $n$.
Likewise, iterating the equality 
$$
S^-(\delta_{1}) \cap S^-(\delta_{r}) = \emptyset
$$
gives that $S^-(f_{\sharp}^{2N}\delta_{1})$ is also disjoint from $S^-(\delta_{r})$.
Thus the connected set 
$$
C = S^-(f_{\sharp}^{2N}\delta_{1}) \cup S^+(\gamma_{1})
$$
is disjoint from $S^-(\delta_{r})$ and $S^+(\gamma_{r})$. 
 Due to hypothesis $(H'_{2})$ about the cyclic order at infinity, $C$ must separates $S^-(\delta_{r})$ from $S^+(\gamma_{r})$.
Since $S^-(f_{\sharp}^{2N}\delta_{r})$ contains the first of this two sets and meets the second one, it must also meet  $C$. As before 
$S^-(f_{\sharp}^{2N}\delta_{r})$ is disjoint from $S^-(f_{\sharp}^{2N}\delta_{1})$, thus $S^-(f_{\sharp}^{2N}\delta_{r})$ meets $S^+(\gamma_{1})$. Iterating positively we get that  
$f_{\sharp}^{n}\delta_{r}$ meets $S^+(\gamma_{1})$ for some $n \geq 2N+1$, which proves the point.
\end{proof}

\subsection{Properties of $T$}
From now on we assume the hypotheses $(H'_{1})$ and $(H'_{2})$ of the theorem.

\begin{lemm}~

\begin{enumerate}
\item If ${t}\in T$ has distinct end-points, then there exists $n >0$ such that $ \mathrm{cut} \underline{f}^n({t})$ contains two distinct elements, also with distinct end-points. 
\item There exists some ${t} \in T$, with distinct end-points, and some $n>0$ such that 
$$
2{t} \in  \mathrm{cut} \underline{f}^n({t}).
$$
\item For such a ${t}\in T$, we have $-{t} \in \mathrm{cut} \underline{f}^n({t})$.
\end{enumerate}

\end{lemm}
\begin{proof}
For the first point, assume that the end-points of the geodesic $\alpha \in {t}$ belongs to  $S^+(\gamma_{i})$ and $S^+(\gamma_{j})$ with $i \neq j$.
Due to the assumption on the cyclic order at infinity, the set 
$$
\cC = \alpha \cup S^+(\gamma_{i}) \cup S^+(\gamma_{j})
$$
separates $S^-(\delta_{k})$ from $S^+(\gamma_{k})$ for some $k \neq i,j$. 
Thus $S^+(f_{\sharp}^{-(2N+1)}\gamma_{k})$ meets $\cC$, but it is disjoint from $S^+(\gamma_{i})$ and $S^+(\gamma_{j})$, thus it must meet $\alpha$. This means that 
$f_{\sharp}^{2N+1} \alpha$ meets $S^+(\gamma_{k})$, and thus the sum $\mathrm{cut}\underline{f}^n({t})$ contains two distinct elements with distinct end-points.

\bigskip

The second point follows from the first one by a purely combinatorial argument. We use the oriented graph $\Gamma_{T}$ whose vertices are the elements of $T$, 
with one edge from ${t_{1}}$ to ${t_{2}}$ for each occurrence of ${t_{2}}$ in $\mathrm{cut} \underline{f}(t_{1})$ (we have already described such a graph for the example in section~\ref{ss.construction-fitted}; note that there may be several edges having the same end-points). The equality $(\mathrm{cut} \circ \underline{f})^n = \mathrm{cut} \circ \underline{f}^n$ have the following nice interpretation: for every ${t_{1}},{t_{2}} \in T$, the number of oriented paths of length $n$ from ${t_{1}}$ to ${t_{2}}$ is equal to the multiplicity of ${t_{2}}$ in 
$\mathrm{cut} \underline{f}^n({t_{1}})$. Denote by $T'$ the subset of $T$ containing the elements with distinct end-points. Thus the first point of the lemma says that for every ${t_{1}} \in T'$ there is at least two distinct paths of the same length from ${t_{1}}$ to some elements of $T'$. 
We call \emph{cycle} a path in $\Gamma_{T}$ starting and ending at the same vertex. Cycles may be indexed by $\bbZ/\ell\bbZ$, and we identify two cycles when they differ from a translation in $\bbZ/\ell\bbZ$. A cycle is called \emph{injective} if the corresponding map $\bbZ/\ell\bbZ \to T$ is injective.
Note that for every cycle $c$, for every element $t$ of $c$, $c$ contains an injective cycle $c'$ containing $t$ (remove inductively loops that do not contain $t$).
The second point of the lemma amounts to finding some ${t} \in T'$ which belongs to two distinct (non necessarily injective) cycles of the same length.
We first prove that there exists some injective cycle $c$ containing some vertex in $T'$ and meeting some other injective cycle $c'$. For this we argue by contradiction. Assume on the contrary that every injective cycle meeting $T'$ is disjoint from every other injective cycle.
From this assumption we get a partial order on the set of injective cycles meeting $T'$, deciding that $d' < d$ if there is a path from some vertex of $d$ to some vertex of $d'$. Consider some injective cycle $d$ meeting $T'$ which is minimal for this order. Choose some $t \in d \cap T'$. Due to the first point there is some path from $t$ to some $t_{1} \in T' \setminus d$. Applying inductively the first point we get an infinite path starting from $t_{1}$ and meeting $T'$ infinitely many times. This path must contain a cycle meeting $T'$, and thus an injective cycle meeting $T'$. This contradicts the minimality of $d$. Thus we have an injective cycle $c$, meeting $T'$ at some vertex $t$, and meeting some other cycle $c'$. Then we may easily construct two different (non injective) cycles starting at $t$ and having the same length: the first one is just $c$ repeated a certain number of times, for the second one 
we run along $c$ from $t$ to the intersection vertex with $c'$, then we run along $c'$ a certain number of times, then we go back to $t$ along the end of $c$; the number of repetitions of $c$ and $c'$ are adjusted to get equal total lengths. This proves the second point.

\bigskip

The argument of the last point is geometric. Let $\alpha$ be a geodesic representing some ${t}$ such that $2{t} \in  \mathrm{cut} \underline{f}^n({t})$.
Let $\alpha'$ be the geodesic in the homotopy class $\underline{f}^n({t})$. 
The geodesic $\alpha$ joins the end-point of $S^+(\gamma_{i_{0}})$ to the end-point of $S^+(\gamma_{i_{1}})$ for some $i_{0}, i_{1}$, and is otherwise disjoint from these topological half-lines. Likewise, $\alpha'$ joins the end-point of $S^+(f_{\sharp}^n \gamma_{i_{0}})$ to the end-point of $S^+(f_{\sharp}^n \gamma_{i_{1}})$ and is otherwise disjoint from these smaller half-lines. (At this point, we suggest that the reader try and draw a picture avoiding the conclusion that $-{t} \in \mathrm{cut} \underline{f}^n({t})$.)
Let $\cT$ be the family of connected components $\tau$ of $\alpha' \setminus S^+$ giving rise to an arc $\tau''$ such that $\underline{\tau}'' \in \mathrm{cut}(\underline{\alpha}')$   and $\tau''$ is homotopic to $\alpha$ (we use the notations of the definition of the map $\mathrm{cut}$). Since $\alpha'$ is a simple arc, distinct elements of $\cT$ are disjoint. By hypothesis $\cT$ contains at least two elements $\tau_{1}, \tau_{2}$, and we assume that $\tau_{2}$ comes after $\tau_{1}$ in the orientation along $\alpha'$, in other words $\alpha'$ is the concatenation of five (possibly degenerate) arcs $\sigma_{1} \tau_{1} \delta \tau_{2} \sigma_{2}$.
Let $\alpha_{0}$ be the arc included in $S^+(\gamma_{0})$ joining $\tau_{1}(0)$ and $\tau_{2}(0)$, and define $\alpha_{1}$ similarly. Denote by $R(\tau_{1}, \tau_{2})$ the closed domain surrounded by the Jordan curve $\tau_{1} \cup \tau_{2} \cup \alpha_{0} \cup \alpha_{1}$; since $\tau''_{1}$ and $\tau''_{2}$ are homotopic, this domain is disjoint from the set $\cO$.
Let $\tau_{3}\in \cT$ and assume that $\tau_{3}$ meets the interior of $R(\tau_{1}, \tau_{2})$. Then $\tau_{3}$ is included in $R(\tau_{1}, \tau_{2})$, and from this we deduce that $R(\tau_{1}, \tau_{3}) \subset R(\tau_{1}, \tau_{2})$. Since $\cT$ is a finite family we may assume that $R(\tau_{1}, \tau_{2})$ is minimal among all the $R(\tau, \tau')$ for distinct $\tau,\tau' \in \cT$: no connected component of $\alpha' \cap \inte R(\tau_{1}, \tau_{2})$ joins a point of $\alpha_{0}$ to a point of $\alpha_{1}$.  

We now argue by contradiction, assuming that $-{t} \not \in \mathrm{cut} f^n({t})$.  In particular no connected component of $\alpha' \cap  \inte  R(\tau_{1}, \tau_{2})$ joins a point of $\alpha_{1}$ to a point of $\alpha_{0}$.  
Thus there exists a simple arc $\beta$ from $\tau_{2}(0)$ to $\tau_{1}(1)$, whose interior is included in the interior of $R(\tau_{1}, \tau_{2})$ and disjoint from $\alpha'$  (to construct such an arc, start from $\tau_{2}(0)$ and follow closely $\alpha_{0}$ from the inside of $R(\tau_{1}, \tau_{2})$ until it meets a connected component of $\alpha' \setminus S^+$, then follow this component, which necessarily joins two points of $\alpha_{0}$, then follow again $\alpha_{0}$, and so on until you arrive near $\tau_{1}(0)$, and finally follow $\tau_{1}$).
We connect the end-points of $\beta$ along $\alpha'$, getting a Jordan curve $\beta \cup \delta$. This curve separates $\tau_{1}(0)$ from $\tau_{2}(1)$, thus one of these two points, say $\tau_{1}(0)$, belongs to the bounded component of $\bbR^2 \setminus (\beta \cup \delta)$. The curve $\sigma_{1}$ is disjoint from $\beta \cup \delta$, thus $\sigma_{1}(0) = \alpha'(0)$ also belongs to this bounded component. On the other hand remember that $\alpha \in \cG_{0}$ is disjoint from $S^+(\gamma_{i_{0}})$ except at its end-points, and thus (since geodesics are in minimal position) $\alpha' = f_{\sharp}^n (\alpha)$ is disjoint from  the properly embedded half-line $S^+(f_{\sharp}^n \gamma_{i_{0}})$ except at its end-points.
Thus $S^+(f_{\sharp}^n \gamma_{i_{0}})$ is disjoint from $R(\tau_{1}, \tau_{2})$ which contains $\beta$. It is also disjoint from $\delta$. The point $\alpha'(0)$ is the end-point of $S^+(f_{\sharp}^n \gamma_{i_{0}})$, it may not be in the bounded component of $\bbR^2 \setminus (\beta \cup \delta)$. This is a contradiction.
\end{proof}

\subsection{Conclusion}
Applying the last lemma provides some $\alpha$ (with distinct end-points) and a positive integer $n$ such that $-\underline{\alpha} \in \mathrm{cut} \underline{f}^n(\underline{\alpha})$. We now apply proposition~\ref{prop.fixed-point} to get a fixed point for $f$. This completes the proof of the theorem.

\bigskip
\bigskip
\bigskip

\section{Orbit diagrams}
\label{sec.orbit-diagrams}

In this section we briefly discuss the possibility of an invariant describing the way orbits of a Brouwer homeomorphism are ``crossing each others''. In other words, we would like to  classify the finite families of orbits of Brouwer homeomorphisms, from the point of view of homotopy Brouwer theory.

The above proposition~\ref{prop.handel-matsumoto} is a special case of the following more general result of Handel (that we will not prove)\footnote{This theorem does not appear explicitly in~\cite{handel1999fixed}, but may be obtained as follows from results in that paper.
Proposition 6.6 provides the existence of the $\delta_{i},\gamma_{i}$ without the ``homotopy disjointness'' required by the last sentence of the theorem.
Then Lemma 4.6 allows to gather the $\gamma_{i}$'s whose forward homotopy streamlines are not disjoint, giving another family $\gamma'_{1}, \dots, \gamma'_{r'}$ of generalized homotopy translation arcs (see the definition in Handel's paper), with $r' \leq r$, each $\gamma'_{i}$ meeting one or several orbits in $\cO$. From $\gamma'_{j}$ one can construct a third family $\gamma''_{1}, \dots, \gamma''_{r}$ such that the forward homotopy streamlines $S^+(\gamma_{i})$ are pairwise disjoint (from each generalized translation arc $\gamma'_{i}$ we construct several homotopy translation arcs which are pairwise disjoint and have representatives inside a small neighborhood of $\gamma'_{i}\cup f(\gamma'_{i})$). Similarly we get a family of pairwise homotopically disjoint backward homotopy streamlines $S^-(\delta''_{i})$. By properness we may choose some integer $N$ such that for every $i$ and every $n \geq 2N$, $f^n(\gamma_{i})$ is homotopically disjoint from $\delta_{j}$. This gives the homotopy disjointness property.}.

\begin{theo*}
Assume $f$ is fixed point free, and let $\cO(x_{1}, \dots , x_{r})$ be the union of finitely many orbits of $f$.
Then property $(H'_{1})$ holds.
\end{theo*}

Assume as above that $f$ is a Brouwer homeomorphism, and that property$(H'_{1})$ is satisfied. 
As before we consider the $2r$ properly embedded half-lines
$$
S^-(\delta_{1}), S^+(\gamma_{1}), \dots , S^-(\delta_{r}), S^+(\gamma_{r}).
$$

As in hypothesis $(H'_{2})$ at the beginning of the previous section,   on this set of pairwise disjoint properly embedded half-lines, we consider the cyclic order at infinity. As for flows, it is convenient to represent this order by placing pairwise distinct points $\alpha_{1}, \omega_{1}, \dots, \alpha_{r}, \omega_{r}$ on a circle and drawing a chord from $\alpha_{i}$ to $\omega_{i}$. Let us denote this diagram by $\cD(f, \delta_{1}, \gamma_{1}, \dots , \delta_{r}, \gamma_{r})$.

We would like this to be an invariant, that is, to depend only on the map $f$ and the points  $x_{1}, \dots, x_{r}$. Unfortunately this is not the case.
Consider the easy case of two orbits $\cO(x_{1}), \cO(x_{2})$ for the translation (first picture in section~\ref{sec.hta}.\ref{ssec.examples}). From these data we may obtain four diagrams, depending on the choice of the family of proper homotopy translation arcs.
In the case of the Reeb flow (second picture in section~\ref{sec.hta}.\ref{ssec.examples}), however, as the homotopy class of translation arcs is unique, we always get the same diagram.

Consider a combinatorial diagram $\cD_{0}$ of oriented chords $[\alpha_{i}, \omega_{i}]$ of the circle. Assume there is two end-points of the same type, say $\alpha_{i}, \alpha_{j}$, which are adjacent in the cyclic order. Then we may obtain a new diagram $\cD_{1}$ by exchanging $\alpha_{i}$ and $\alpha_{j}$ in the cyclic order. We will say that $\cD_{1}$  is obtain from $\cD_{0}$ by an \emph{elementary operation}.
It can be proved that if $\cD_{0} = \cD(f, \delta_{1}, \gamma_{1}, \dots , \delta_{r}, \gamma_{r})$ is some diagram for $(f, x_{1}, \dots , x_{r})$, then any diagram obtained 
from $\cD_{0}$ by performing a sequence of elementary operations is a diagram $\cD(f, \delta'_{1}, \gamma'_{1}, \dots , \delta'_{r}, \gamma'_{r})$ for the same points (but for different choices of homotopy classes of homotopy translation arcs). 

\begin{exercise}
Consider the Reeb map with $r=3$, as in the examples of section~\ref{sec.hta}. Choose a family of homotopy translation arcs as in hypothesis $(H'_{1})$, draw the associated diagram. Perform an elementary operation on this diagram, and find another  family of homotopy translation arcs, still satisfying hypothesis $(H'_{1})$, and corresponding to this new diagram. Do this for the four possible diagrams.
\end{exercise}

Conversely, we may conjecture that elementary operations allow to describe all possible diagrams associated to $(f, x_{1}, \dots , x_{r})$. Let us put this another way. Consider again some abstract diagram $\cD_{0}$. The \emph{reduced diagram} associated to $\cD_{0}$, say $\cD_{0}^R$, is obtained from $\cD_{0}$ by identifying all the vertices of the same type ($\alpha$ or $\omega$) that are adjacent.  
For example, for the translation with two orbits, starting from any of the four diagrams we get as a reduced diagram the diagram with a single chord of multiplicity two, whereas for the Reeb case, the reduced diagram coincides with the unreduced diagram.
The conjecture says that  given two different choices of homotopy translation arcs 
$$
\delta_{1}, \gamma_{1}, \dots , \delta_{r}, \gamma_{r} \mbox{ and }
\delta'_{1}, \gamma'_{1}, \dots , \delta'_{r}, \gamma'_{r}
$$
associated to the same data $(f, x_{1}, \dots , x_{r})$, the reduced diagrams coincides,
$$
\cD(f, \delta_{1}, \gamma_{1}, \dots , \delta_{r}, \gamma_{r})^R = 
\cD(f, \delta'_{1}, \gamma'_{1}, \dots , \delta'_{r}, \gamma'_{r})^R. 
$$
If the conjecture holds, then the reduced diagram is an invariant of homotopy Brouwer theory associated to $(f, x_{1}, \dots , x_{r})$. This invariant would describe in a natural way ``the way that the orbits crosses each others.''
Another (probably much harder) conjecture says that this a total invariant. In other words, assume that the two sets of data $(f, x_{1}, \dots , x_{r})$ and $(f', x'_{1}, \dots , x'_{r})$ give rise to the same reduced diagram. Then the data should be equivalent from homotopy Brouwer theory viewpoint, which means that
there exists a homeomorphisms $\Phi$ that sends each point $x_{i}$ on the point $x'_{i}$, and such that the homeomorphisms $\Phi f \Phi^{-1}$ and $f'$ are isotopic relative to $\cO(x'_{1}, \dots , x'_{r})$ (the ``braid types'' are the same).


\bigskip
\bigskip
\bigskip

\section{Appendix 1: geodesics}

 We consider a locally finite countable subset $\cO$ of the plane, and the set $\cA$ of \emph{essential simple curves} (see section~\ref{sec.hta}.\ref{ss.hta-def}). 
In this text we make use of the existence of a subset $\cH$ of elements of the set $\cA$ called \emph{geodesics}
with the following properties. Two  curves $\alpha,\beta \in \cA$ 
 are in \emph{minimal position} if they are topologically transverse\footnote{In the neighbourhood of every intersection point, up to a homeomorphism, $\alpha$ is a vertical segment and $\beta$ is a horizontal segment.} and every connected component of $\bbR^2 \setminus (\alpha \cup \beta)$ whose boundary is made of exactly one piece of $\alpha$ and one piece of $\beta$ contains at least one element of $\cO$.
\begin{exercise}
Prove that $\alpha$ and $\beta$ are in minimal position if and only if for every $\alpha',\beta'$ homotopic respectively  to $\alpha,\beta$,  
$$
\sharp \alpha' \cap \beta' \geq \sharp \alpha \cap \beta. 
$$ 
Hint: use the universal cover of $\bbR^2 \setminus \cO$, and prove that, when they are in minimal position, $\sharp \alpha \cap \beta$ is equal to the number of lifts of $\beta$ that separates the  beginning and the end of a lift of $\alpha$.
\end{exercise}


 A family $\{\alpha_{n}\}$ of curves is said to be \emph{proper} (or \emph{locally finite}) if every compact set $K$ meets only a finite number of $\alpha_{n}$'s.
We denote by $\mathrm{Homeo}_{0}(\bbR^2, \cO)$ the connected component of the identity within the space of homeomorphisms of the plane that fixe $\cO$ point-wise (an element of this group is said to be isotopic to the identity relative to $\cO$).
\begin{enumerate} 
\item Each homotopy class in $\cA$ contains a unique element of $\cH$, that is, the map $\alpha \mapsto \underline{\alpha}$ from $\cH$ to $\underline \cA$ is one-to-one and onto. 

\item Every couple of curves $\alpha,\beta \in \cH$ with $\alpha \neq \pm \beta$
 is in minimal position. In particular, if $\alpha$ and $\beta$ are homotopically disjoint then $\alpha \cap \beta \subset \cO$.

\item Let $\{\alpha_{i}\}$ be an at most countable family of pairwise non-homotopic and disjoint curves in $\cA$ 
 which is locally finite. Then there exists $h \in \mathrm{Homeo}_{0}(\bbR^2, \cO)$ such that all the $h(\alpha_{i})$'s belong to $\cH$.

\item 
More generally, let $\{\alpha_{i}\}$ be as in the previous item, and 
let $\{\beta_{j}\}$ having the same properties.  Assume that every curve $\alpha_{i}$ is  non-homotopic to and in minimal position with every curve $\beta_{j}$. Then there exists $h \in \mathrm{Homeo}_{0}(\bbR^2, \cO)$ such that all the $h(\alpha_{i})$'s, $h(\beta_{j})$'s belong to $\cH$.

\item Let  $(\alpha_{n})_{n \geq 0}$ be a sequence  in $\cA$,
 and for every $n$ let $\alpha'_{n}$ be the element of $\cH$
   homotopic to $\alpha_{n}$. If  $(\alpha_{n})_{n \geq 0}$ is homotopically proper then  $(\alpha'_{n})_{n \geq0}$ is proper.

\end{enumerate}
\begin{exercise}
Prove  that the last item (given the firsts) is equivalent to the following property. Number the elements of $\cO$ so that $\cO = \{u_{n}, n \geq 0\}$.
Let $(D_{n})_{n \geq 0}$ be the increasing sequence of topological disks with geodesic boundary (\emph{i.e.} the curves $\partial D_{n}$ belong to $\cH'$). Then the sequence 
$(\partial D_{n})_{n \geq 0}$ is proper.
\end{exercise}

We also make use (in the proof of the crucial proposition~\ref{prop.fixed-point}) of the existence of a circle compactification of the universal cover of $\bbR^2 \setminus \cO$ with nice properties. Let $\bbH^2$ denotes the open unit disk with boundary $\partial \bbH^2 = \bbS^1$.
\pagebreak
\begin{theo*}
There exists a (universal) covering $\pi : \bbH^2 \to \bbR^2 \setminus \cO$ with the following properties.
\begin{enumerate}
\item Every lift $\tilde h $ of a homeomorphism $h : \bbR^2 \setminus \cO \to \bbR^2 \setminus \cO$ extends to a homeomorphism of $\bbH^2 \cup \partial \bbH^2$. 
Furthermore if there is an isotopy $I$ relative to $\cO$ from $h$ to $h'$ (that is, a path $(\Phi_{t} \circ h)$ where $(\Phi_{t})_{t \in [0,1]}$ is a path in $\mathrm{Homeo}_{0}(\bbR^2, \cO)$ starting at the identity), and if we lift $I$ to get an isotopy from a lift $\tilde h$ to a lift $\tilde h'$, then $\tilde h$ and $\tilde h'$ have the same extension to the boundary.
\item Every lift $\tilde \alpha$ of a curve $\alpha$ in $\cA$ 
 (or more precisely of the restriction of $\alpha$ to $(0,1)$) 
 admits end-points $\lim_{t \to 0} \alpha(t)$ and $\lim_{t \to 1} \alpha(t)$ in $\partial \bbD^2$.  Furthermore if there is an isotopy $I$ relative to $\cO$ from $\alpha$ to $\alpha'$, and if we lift $I$ to get an isotopy from a lift $\tilde \alpha$ to a lift $\tilde \alpha'$, then $\tilde \alpha$ and $\tilde \alpha'$ have the same end-points. 
\end{enumerate}
\end{theo*}
We will abuse notation and denotes the end-points of a lift $\tilde \alpha$  by $\tilde \alpha(0)$ and $\tilde \alpha(1)$.

\bigskip

The key to these properties is the existence of a hyperbolic structure on the surface $\bbR^2 \setminus \cO$. This structure may be obtained by gluing together two copies of an infinite-sided hyperbolic polygon (see~\cite{handel1999fixed}, section~3). Thus $\bbR^2 \setminus \cO$ appears as the quotient of the hyperbolic plane $\bbH^2$ by a discrete group $\Gamma$ of hyperbolic isometries acting freely. A crucial point here, as explicited by Matsumoto in~\cite{matsumoto2000arnold}, is that $\Gamma$ is \emph{of the first type}, which means that the orbit $\Gamma.x$ of any point $x \in \bbH^2$ accumulates on the whole boundary $\partial \bbH^2$. This property is involved in particular in the uniqueness and properness of geodesics (property 1 and 5 in the above list); the above extension theorem holds only under this additional property. The references for the proofs are the book by Bleiler and Casson (\cite{cassonbleiler}) and the previously quoted paper by Matsumoto.

\section{Appendix 2: pictures of a fitted family}
~

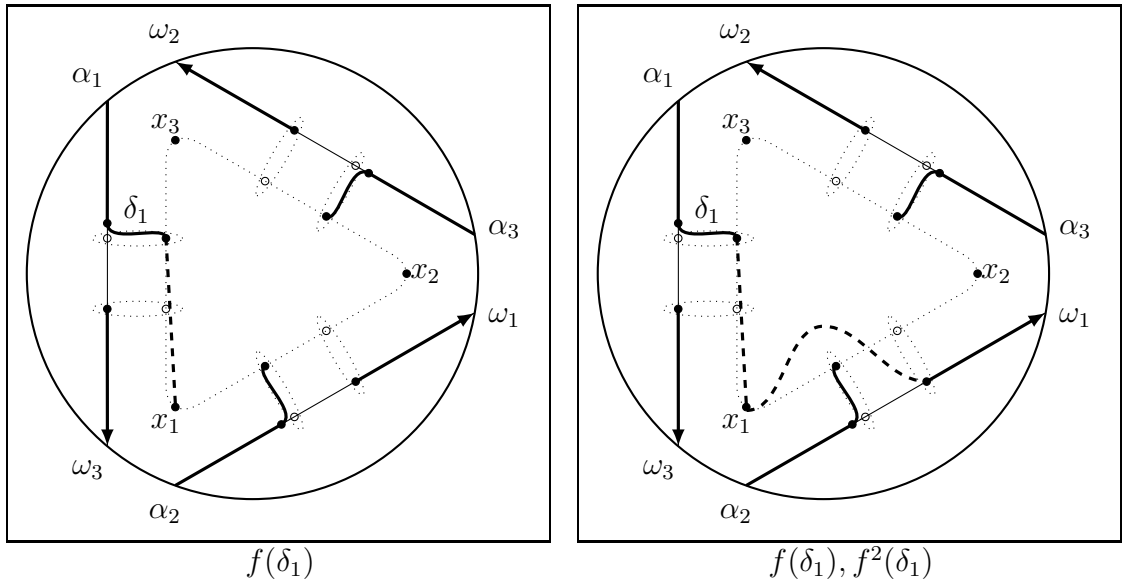
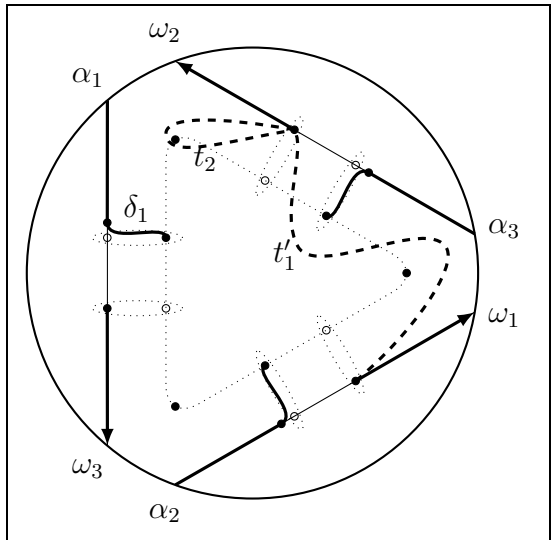
\begin{minipage}{0.5\textwidth}
\begin{center}
\fbox{
\begin{tikzpicture}[scale=1] 
\tikzstyle{fleche}=[>=latex,->]
\draw [thick] (0,0) circle (3cm) ;
\foreach \i in {1,2,3}  
{\draw  (\i*120+10:3) -- (\i*120+110:3)  ;
\draw (\i*120+10:3.4) node {$\alpha_{\i}$};
\draw  (\i*120+230:3.4) node {$\omega_{\i}$};
\draw [dotted,rotate=\i*120-120]   (43+120:1.6) ellipse (0.6 and 0.1);
\draw [dotted,rotate=\i*120-120]   (-43-120:1.6) ellipse (0.6 and 0.1);

\filldraw [rotate=\i*120] (-1.15,0.47) circle (0.05) ;
\draw [rotate=\i*120] (-1.93,0.47) circle (0.05) ;

\draw [rotate=\i*120] (-1.15,-0.47) circle (0.05) ;
\filldraw [rotate=\i*120] (-1.93,-0.47) circle (0.05) ;

\filldraw [rotate=\i*120] (0:2.05) circle (0.05) ;
\draw [rotate=\i*120+120] (0:2.3) node {$x_{\i}$} ;

\draw [rotate=\i*120,very thick] (1*120+10:3) -- (-1.93,0.67)  ;
\filldraw [rotate=\i*120]  (-1.93,0.67) circle (0.05) ;
\draw [rotate=\i*120,very thick] (-1.93,0.67) .. controls +(0,-0.3) and +(0,0.2) .. (-1.15,0.47) ;
\draw [rotate=\i*120,very thick,fleche] (-1.93,-0.47)  -- (1*120+110:3) ; 
}

\def \k {2.3}
\draw [rounded corners=0.5cm, dotted] (0:\k) -- (120:\k) -- (240:\k) -- cycle;

\foreach \i in {1}  
{

\path (-1.93,0.67) .. controls +(0,-0.3) and +(0,0.2) .. (-1.15,0.47) node [pos=0.5, above] {$\delta_{1}$} ;
\draw [dashed, very thick] (-1.15,0.48)  -- (-1.02,-1.82);
}
\end{tikzpicture} 
}
$f(\delta_{1})$
\end{center}
\end{minipage}
\begin{minipage}{0.5\textwidth}
\begin{center}
\fbox{
\begin{tikzpicture}[scale=1] 
\tikzstyle{fleche}=[>=latex,->]
\draw [thick] (0,0) circle (3cm) ;
\foreach \i in {1,2,3}  
{\draw  (\i*120+10:3) -- (\i*120+110:3)  ;
\draw (\i*120+10:3.4) node {$\alpha_{\i}$};
\draw  (\i*120+230:3.4) node {$\omega_{\i}$};
\draw [dotted,rotate=\i*120-120]   (43+120:1.6) ellipse (0.6 and 0.1);
\draw [dotted,rotate=\i*120-120]   (-43-120:1.6) ellipse (0.6 and 0.1);

\filldraw [rotate=\i*120] (-1.15,0.47) circle (0.05) ;
\draw [rotate=\i*120] (-1.93,0.47) circle (0.05) ;

\draw [rotate=\i*120] (-1.15,-0.47) circle (0.05) ;
\filldraw [rotate=\i*120] (-1.93,-0.47) circle (0.05) ;

\filldraw [rotate=\i*120] (0:2.05) circle (0.05) ;
\draw [rotate=\i*120+120] (0:2.3) node {$x_{\i}$} ;

\draw [rotate=\i*120,very thick] (1*120+10:3) -- (-1.93,0.67)  ;
\filldraw [rotate=\i*120]  (-1.93,0.67) circle (0.05) ;
\draw [rotate=\i*120,very thick] (-1.93,0.67) .. controls +(0,-0.3) and +(0,0.2) .. (-1.15,0.47) ;
\draw [rotate=\i*120,very thick,fleche] (-1.93,-0.47)  -- (1*120+110:3) ; 
}

\def \k {2.3}
\draw [rounded corners=0.5cm, dotted] (0:\k) -- (120:\k) -- (240:\k) -- cycle;

\foreach \i in {1}  
{

\path (-1.93,0.67) .. controls +(0,-0.3) and +(0,0.2) .. (-1.15,0.47) node [pos=0.5, above] {$\delta_{1}$} ;
\draw [dashed, very thick] (-1.15,0.48)  -- (-1.02,-1.82);
\draw [dashed, very thick]  (-1.02,-1.82) .. controls +(0:0.5) and +(180:0.5) .. (0,-0.7) .. controls +(0:0.5) and +(180:0.5) .. (1.35,-1.43) ;
}
\end{tikzpicture} 
}\\
$f(\delta_{1}), f^2(\delta_{1})$
\end{center}
\end{minipage}

\bigskip

\begin{minipage}{0.5\textwidth}
\begin{center}
\fbox{
\begin{tikzpicture}[scale=1] 
\tikzstyle{fleche}=[>=latex,->]
\draw [thick] (0,0) circle (3cm) ;
\foreach \i in {1,2,3}  
{\draw  (\i*120+10:3) -- (\i*120+110:3)  ;
\draw (\i*120+10:3.4) node {$\alpha_{\i}$};
\draw  (\i*120+230:3.4) node {$\omega_{\i}$};
\draw [dotted,rotate=\i*120-120]   (43+120:1.6) ellipse (0.6 and 0.1);
\draw [dotted,rotate=\i*120-120]   (-43-120:1.6) ellipse (0.6 and 0.1);

\filldraw [rotate=\i*120] (-1.15,0.47) circle (0.05) ;
\draw [rotate=\i*120] (-1.93,0.47) circle (0.05) ;

\draw [rotate=\i*120] (-1.15,-0.47) circle (0.05) ;
\filldraw [rotate=\i*120] (-1.93,-0.47) circle (0.05) ;

\filldraw [rotate=\i*120] (0:2.05) circle (0.05) ;

\draw [rotate=\i*120,very thick] (1*120+10:3) -- (-1.93,0.67)  ;
\filldraw [rotate=\i*120]  (-1.93,0.67) circle (0.05) ;
\draw [rotate=\i*120,very thick] (-1.93,0.67) .. controls +(0,-0.3) and +(0,0.2) .. (-1.15,0.47) ;
\draw [rotate=\i*120,very thick,fleche] (-1.93,-0.47)  -- (1*120+110:3) ; 
}

\def \k {2.3}
\draw [rounded corners=0.5cm, dotted] (0:\k) -- (120:\k) -- (240:\k) -- cycle;

\foreach \i in {1}  
{

\path (-1.93,0.67) .. controls +(0,-0.3) and +(0,0.2) .. (-1.15,0.47) node [pos=0.5, above] {$\delta_{1}$} ;
\draw [dashed, very thick] (-1.15,0.48)  -- (-1.02,-1.82);
\draw [dashed, very thick]  (-1.02,-1.82) .. controls +(0:0.5) and +(180:0.5) .. (0,-0.7) .. controls +(0:0.5) and +(180:0.5) .. (1.35,-1.43) ;
\draw  [dashed, very thick,rotate=\i*120]  (-1.93,-0.87) .. controls +(-50:1) and +(-135:0.3) ..   (-1,-1.9) 
.. controls +(45:0.3) and +(-50:1) .. (-1.93,-0.47)  ;
\draw [rotate=120] (-1.93,-0.87) circle (0.05) ;
}
\end{tikzpicture} 
}\\
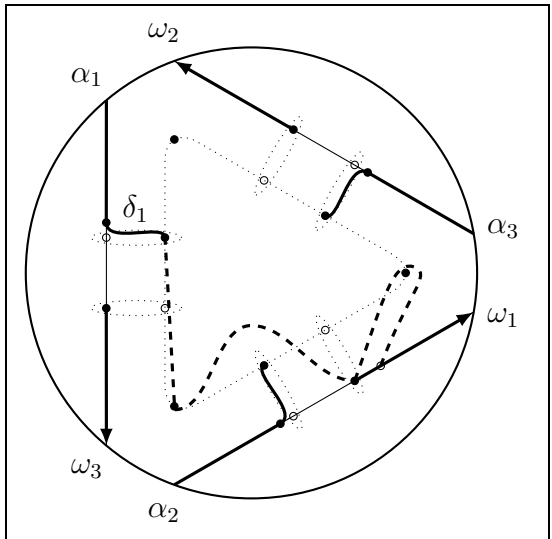
$f(\delta_{1}), f^2(\delta_{1}), f^3(\delta_{1})$
\end{center}
\end{minipage}
\begin{minipage}{0.5\textwidth}
\begin{center}
\fbox{
\begin{tikzpicture}[scale=1] 
\tikzstyle{fleche}=[>=latex,->]
\draw [thick] (0,0) circle (3cm) ;
\foreach \i in {1,2,3}  
{\draw  (\i*120+10:3) -- (\i*120+110:3)  ;
\draw (\i*120+10:3.4) node {$\alpha_{\i}$};
\draw  (\i*120+230:3.4) node {$\omega_{\i}$};
\draw [dotted,rotate=\i*120-120]   (43+120:1.6) ellipse (0.6 and 0.1);
\draw [dotted,rotate=\i*120-120]   (-43-120:1.6) ellipse (0.6 and 0.1);

\filldraw [rotate=\i*120] (-1.15,0.47) circle (0.05) ;
\draw [rotate=\i*120] (-1.93,0.47) circle (0.05) ;

\draw [rotate=\i*120] (-1.15,-0.47) circle (0.05) ;
\filldraw [rotate=\i*120] (-1.93,-0.47) circle (0.05) ;

\filldraw [rotate=\i*120] (0:2.05) circle (0.05) ;

\draw [rotate=\i*120,very thick] (1*120+10:3) -- (-1.93,0.67)  ;
\filldraw [rotate=\i*120]  (-1.93,0.67) circle (0.05) ;
\draw [rotate=\i*120,very thick] (-1.93,0.67) .. controls +(0,-0.3) and +(0,0.2) .. (-1.15,0.47) ;
\draw [rotate=\i*120,very thick,fleche] (-1.93,-0.47)  -- (1*120+110:3) ; 
}

\def \k {2.3}
\draw [rounded corners=0.5cm, dotted] (0:\k) -- (120:\k) -- (240:\k) -- cycle;

\foreach \i in {1}  
{

\path (-1.93,0.67) .. controls +(0,-0.3) and +(0,0.2) .. (-1.15,0.47) node [pos=0.5, above] {$\delta_{1}$} ;
\draw  [dashed, very thick,rotate=\i*120]  (-1.93,-0.87) .. controls +(-50:1) and +(-135:0.3) ..   (-1,-1.9) 
.. controls +(45:0.3) and +(-50:1) .. (-1.93,-0.47)  ;
\draw [rotate=120] (-1.93,-0.87) circle (0.05) ;
}
\end{tikzpicture} 
}\\
$f^3(\delta_{1})$
\end{center}
\end{minipage}

\bigskip

\begin{minipage}{0.5\textwidth}
   \begin{center}
\fbox{
\begin{tikzpicture}[scale=1] 
\tikzstyle{fleche}=[>=latex,->]
\draw [thick] (0,0) circle (3cm) ;
\foreach \i in {1,2,3}  
{\draw  (\i*120+10:3) -- (\i*120+110:3)  ;
\draw (\i*120+10:3.4) node {$\alpha_{\i}$};
\draw  (\i*120+230:3.4) node {$\omega_{\i}$};
\draw [dotted,rotate=\i*120-120]   (43+120:1.6) ellipse (0.6 and 0.1);
\draw [dotted,rotate=\i*120-120]   (-43-120:1.6) ellipse (0.6 and 0.1);

\filldraw [rotate=\i*120] (-1.15,0.47) circle (0.05) ;
\draw [rotate=\i*120] (-1.93,0.47) circle (0.05) ;

\draw [rotate=\i*120] (-1.15,-0.47) circle (0.05) ;
\filldraw [rotate=\i*120] (-1.93,-0.47) circle (0.05) ;

\filldraw [rotate=\i*120] (0:2.05) circle (0.05) ;

\draw [rotate=\i*120,very thick] (1*120+10:3) -- (-1.93,0.67)  ;
\filldraw [rotate=\i*120]  (-1.93,0.67) circle (0.05) ;
\draw [rotate=\i*120,very thick] (-1.93,0.67) .. controls +(0,-0.3) and +(0,0.2) .. (-1.15,0.47) ;
\draw [rotate=\i*120,very thick,fleche] (-1.93,-0.47)  -- (1*120+110:3) ; 
}

\def \k {2.3}
\draw [rounded corners=0.5cm, dotted] (0:\k) -- (120:\k) -- (240:\k) -- cycle;

\foreach \i in {1}  
{

\path (-1.93,0.67) .. controls +(0,-0.3) and +(0,0.2) .. (-1.15,0.47) node [pos=0.5, above] {$\delta_{1}$} ;
\draw [rotate=120] (-1,-1.3) node {$t_1$};
\draw  [dashed, very thick,rotate=\i*120]  (-1.93,-0.47) .. controls +(-70:1) and +(-135:0.3) ..   (-1,-1.9) 
.. controls +(45:0.3) and +(-50:1) .. (-1.93,-0.47) ;

}
\end{tikzpicture} 
}\\

$t_{1} := \mathrm{cut}  f^3(\delta_{1})$

\end{center}
\end{minipage}
\begin{minipage}{0.5\textwidth}
\begin{center}
\fbox{
\begin{tikzpicture}[scale=1] 
\tikzstyle{fleche}=[>=latex,->]
\draw [thick] (0,0) circle (3cm) ;
\foreach \i in {1,2,3}  
{\draw  (\i*120+10:3) -- (\i*120+110:3)  ;
\draw (\i*120+10:3.4) node {$\alpha_{\i}$};
\draw  (\i*120+230:3.4) node {$\omega_{\i}$};
\draw [dotted,rotate=\i*120-120]   (43+120:1.6) ellipse (0.6 and 0.1);
\draw [dotted,rotate=\i*120-120]   (-43-120:1.6) ellipse (0.6 and 0.1);

\filldraw [rotate=\i*120] (-1.15,0.47) circle (0.05) ;
\draw [rotate=\i*120] (-1.93,0.47) circle (0.05) ;

\draw [rotate=\i*120] (-1.15,-0.47) circle (0.05) ;
\filldraw [rotate=\i*120] (-1.93,-0.47) circle (0.05) ;

\filldraw [rotate=\i*120] (0:2.05) circle (0.05) ;

\draw [rotate=\i*120,very thick] (1*120+10:3) -- (-1.93,0.67)  ;
\filldraw [rotate=\i*120]  (-1.93,0.67) circle (0.05) ;
\draw [rotate=\i*120,very thick] (-1.93,0.67) .. controls +(0,-0.3) and +(0,0.2) .. (-1.15,0.47) ;
\draw [rotate=\i*120,very thick,fleche] (-1.93,-0.47)  -- (1*120+110:3) ; 
}

\def \k {2.3}
\draw [rounded corners=0.5cm, dotted] (0:\k) -- (120:\k) -- (240:\k) -- cycle;

\foreach \i in {1}  
{

\path (-1.93,0.67) .. controls +(0,-0.3) and +(0,0.2) .. (-1.15,0.47) node [pos=0.5, above] {$\delta_{1}$} ;
\draw [rotate=120] (-1,-1.3) node {$t_1$};
\draw  [dashed, very thick,rotate=\i*120]  (-1.93,-0.47) .. controls +(-70:1) and +(-135:0.3) ..   (-1,-1.9) 
.. controls +(45:0.3) and +(-50:1) .. (-1.93,-0.47)  ;

\draw [dashed, very thick,rotate=120]   (-1.93,-0.87)  
.. controls +(-70:1) and +(180:0.3) .. (-1,-2.3) 
.. controls +(0:0.5) and +(180:0.5) .. (0,-0.7) 
.. controls +(0:0.5) and +(90:0.5) .. (1.5,-1.5)
.. controls +(-90:0.5) and +(0:0.5) .. (0,-1) 
.. controls +(180:0.5) and +(0:0.5) .. (-1,-2.6) 
.. controls +(180:0.5) and +(-80:1) .. (-1.93,-0.87);
\draw [rotate=120] (-1.93,-0.87) circle (0.05) ;
}
\end{tikzpicture} 
}\\

$f(t_{1}) = f  \mathrm{cut}  f^3(\delta_{1})$
\end{center}
\end{minipage}

\bigskip
\enlargethispage{2cm}
\begin{minipage}{0.5\textwidth}
   \begin{center}
\fbox{
\begin{tikzpicture}[scale=1] 
\tikzstyle{fleche}=[>=latex,->]
\draw [thick] (0,0) circle (3cm) ;
\foreach \i in {1,2,3}  
{\draw  (\i*120+10:3) -- (\i*120+110:3)  ;
\draw (\i*120+10:3.4) node {$\alpha_{\i}$};
\draw  (\i*120+230:3.4) node {$\omega_{\i}$};
\draw [dotted,rotate=\i*120-120]   (43+120:1.6) ellipse (0.6 and 0.1);
\draw [dotted,rotate=\i*120-120]   (-43-120:1.6) ellipse (0.6 and 0.1);

\filldraw [rotate=\i*120] (-1.15,0.47) circle (0.05) ;
\draw [rotate=\i*120] (-1.93,0.47) circle (0.05) ;

\draw [rotate=\i*120] (-1.15,-0.47) circle (0.05) ;
\filldraw [rotate=\i*120] (-1.93,-0.47) circle (0.05) ;

\filldraw [rotate=\i*120] (0:2.05) circle (0.05) ;

\draw [rotate=\i*120,very thick] (1*120+10:3) -- (-1.93,0.67)  ;
\filldraw [rotate=\i*120]  (-1.93,0.67) circle (0.05) ;
\draw [rotate=\i*120,very thick] (-1.93,0.67) .. controls +(0,-0.3) and +(0,0.2) .. (-1.15,0.47) ;
\draw [rotate=\i*120,very thick,fleche] (-1.93,-0.47)  -- (1*120+110:3) ; 
}

\def \k {2.3}
\draw [rounded corners=0.5cm, dotted] (0:\k) -- (120:\k) -- (240:\k) -- cycle;

\foreach \i in {1}  
{

\path (-1.93,0.67) .. controls +(0,-0.3) and +(0,0.2) .. (-1.15,0.47) node [pos=0.5, above] {$\delta_{1}$} ;
\draw [rotate=120] (0,-0.5) node {$t'_{1}$};
\draw  [dashed, very thick,rotate=120]  (-1.93,-0.47)  
.. controls +(-80:1) and +(180:0.3) .. (-1,-2.4) 
.. controls +(0:0.5) and +(180:0.5) .. (0,-0.7) 
.. controls +(0:0.5) and +(180:0.3) .. (1.35,-1.43);
}
\end{tikzpicture} 
}\\

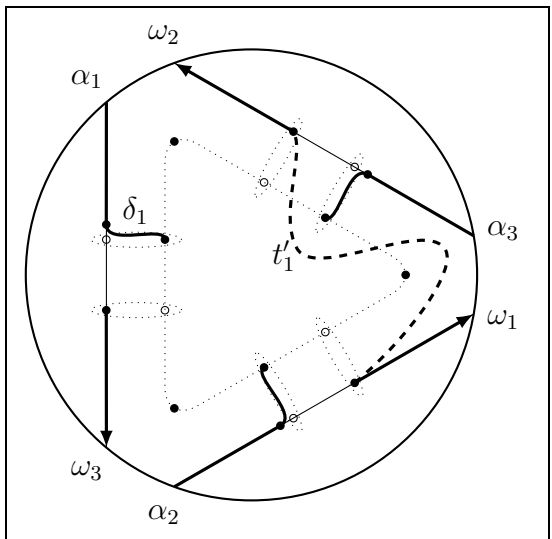
$\mathrm{cut}  f(t_{1}) = t'_{1}+ (-t'_{1})$
\end{center}
\end{minipage}
\begin{minipage}{0.5\textwidth}
\begin{center}
\fbox{
\begin{tikzpicture}[scale=1] 
\tikzstyle{fleche}=[>=latex,->]
\draw [thick] (0,0) circle (3cm) ;
\foreach \i in {1,2,3}  
{\draw  (\i*120+10:3) -- (\i*120+110:3)  ;
\draw (\i*120+10:3.4) node {$\alpha_{\i}$};
\draw  (\i*120+230:3.4) node {$\omega_{\i}$};
\draw [dotted,rotate=\i*120-120]   (43+120:1.6) ellipse (0.6 and 0.1);
\draw [dotted,rotate=\i*120-120]   (-43-120:1.6) ellipse (0.6 and 0.1);

\filldraw [rotate=\i*120] (-1.15,0.47) circle (0.05) ;
\draw [rotate=\i*120] (-1.93,0.47) circle (0.05) ;

\draw [rotate=\i*120] (-1.15,-0.47) circle (0.05) ;
\filldraw [rotate=\i*120] (-1.93,-0.47) circle (0.05) ;

\filldraw [rotate=\i*120] (0:2.05) circle (0.05) ;

\draw [rotate=\i*120,very thick] (1*120+10:3) -- (-1.93,0.67)  ;
\filldraw [rotate=\i*120]  (-1.93,0.67) circle (0.05) ;
\draw [rotate=\i*120,very thick] (-1.93,0.67) .. controls +(0,-0.3) and +(0,0.2) .. (-1.15,0.47) ;
\draw [rotate=\i*120,very thick,fleche] (-1.93,-0.47)  -- (1*120+110:3) ; 
}

\def \k {2.3}
\draw [rounded corners=0.5cm, dotted] (0:\k) -- (120:\k) -- (240:\k) -- cycle;

\foreach \i in {1}  
{
\path (-1.93,0.67) .. controls +(0,-0.3) and +(0,0.2) .. (-1.15,0.47) node [pos=0.5, above] {$\delta_{1}$} ;
\draw [rotate=120,dashed, very thick]   (-1.93,-0.87) 
 .. controls +(-70:1) and +(180:0.3) .. (-1,-2.7) 
.. controls +(0:0.5) and +(180:0.5) .. (0,-0.9) 
.. controls +(0:0.5) and +(180:0.3) .. (1.4,-1.7)
 .. controls +(0:0.3) and    +(180:0.5) ..  (2.2,0.3) 
.. controls +(0:0.5) and  +(45:0.5) ..  (1.72,-1.22) ;
\draw [rotate=120]  (1.72,-1.22) circle (0.05) ; 

\draw [rotate=120] (-1.93,-0.87) circle (0.05) ;

}
\end{tikzpicture} 
}\\

$f(t'_{1})$
\end{center}
\end{minipage}

\bigskip

\begin{minipage}{0.5\textwidth}
   \begin{center}
\fbox{
\begin{tikzpicture}[scale=1] 
\tikzstyle{fleche}=[>=latex,->]
\draw [thick] (0,0) circle (3cm) ;
\foreach \i in {1,2,3}  
{\draw  (\i*120+10:3) -- (\i*120+110:3)  ;
\draw (\i*120+10:3.4) node {$\alpha_{\i}$};
\draw  (\i*120+230:3.4) node {$\omega_{\i}$};
\draw [dotted,rotate=\i*120-120]   (43+120:1.6) ellipse (0.6 and 0.1);
\draw [dotted,rotate=\i*120-120]   (-43-120:1.6) ellipse (0.6 and 0.1);

\filldraw [rotate=\i*120] (-1.15,0.47) circle (0.05) ;
\draw [rotate=\i*120] (-1.93,0.47) circle (0.05) ;

\draw [rotate=\i*120] (-1.15,-0.47) circle (0.05) ;
\filldraw [rotate=\i*120] (-1.93,-0.47) circle (0.05) ;

\filldraw [rotate=\i*120] (0:2.05) circle (0.05) ;

\draw [rotate=\i*120,very thick] (1*120+10:3) -- (-1.93,0.67)  ;
\filldraw [rotate=\i*120]  (-1.93,0.67) circle (0.05) ;
\draw [rotate=\i*120,very thick] (-1.93,0.67) .. controls +(0,-0.3) and +(0,0.2) .. (-1.15,0.47) ;
\draw [rotate=\i*120,very thick,fleche] (-1.93,-0.47)  -- (1*120+110:3) ; 
}

\def \k {2.3}
\draw [rounded corners=0.5cm, dotted] (0:\k) -- (120:\k) -- (240:\k) -- cycle;

\foreach \i in {1}  
{

\path (-1.93,0.67) .. controls +(0,-0.3) and +(0,0.2) .. (-1.15,0.47) node [pos=0.5, above] {$\delta_{1}$} ;

\draw [rotate=120] (0,-0.5) node {$t'_{1}$};
\draw  [dashed, very thick,rotate=120]  (-1.93,-0.47)  
.. controls +(-80:1) and +(180:0.3) .. (-1,-2.4) 
.. controls +(0:0.5) and +(180:0.5) .. (0,-0.7) 
.. controls +(0:0.5) and +(180:0.3) .. (1.35,-1.43);
\draw [rotate=240] (-1,-1.3) node {$t_2$};
\draw  [dashed, very thick,rotate=2400]  (-1.93,-0.47) .. controls +(-70:1) and +(-135:0.3) ..   (-1,-1.9) 
.. controls +(45:0.3) and +(-50:1) .. (-1.93,-0.47)  ;
}
\end{tikzpicture} 
}\\

$\mathrm{cut}  f(t'_{1}) = t'_{1}+t_{2}$\\~
\end{center}
\end{minipage}
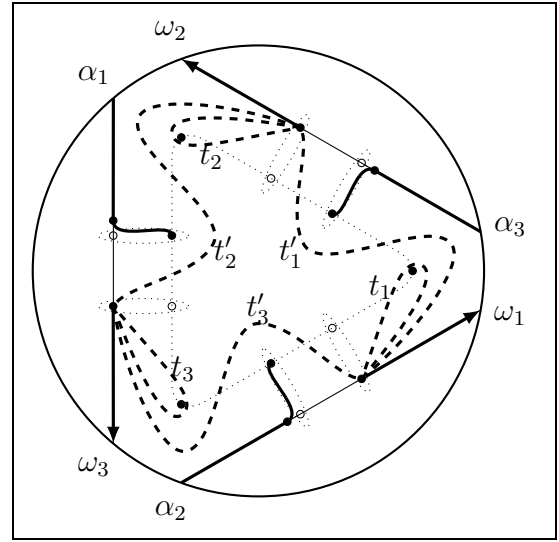
\begin{minipage}{0.5\textwidth}
\begin{center}
\fbox{
\begin{tikzpicture}[scale=1] 
\tikzstyle{fleche}=[>=latex,->]
\draw [thick] (0,0) circle (3cm) ;

\foreach \i in {1,2,3}  
{\draw  (\i*120+10:3) -- (\i*120+110:3)  ;
\draw (\i*120+10:3.4) node {$\alpha_{\i}$};
\draw  (\i*120+230:3.4) node {$\omega_{\i}$};
\draw [dotted,rotate=\i*120-120]   (43+120:1.6) ellipse (0.6 and 0.1);
\draw [dotted,rotate=\i*120-120]   (-43-120:1.6) ellipse (0.6 and 0.1);
\filldraw [rotate=\i*120] (-1.15,0.47) circle (0.05) ;
\draw [rotate=\i*120] (-1.93,0.47) circle (0.05) ;

\draw [rotate=\i*120] (-1.15,-0.47) circle (0.05) ;
\filldraw [rotate=\i*120] (-1.93,-0.47) circle (0.05) ;

\filldraw [rotate=\i*120] (0:2.05) circle (0.05) ;

\draw [rotate=\i*120,very thick] (1*120+10:3) -- (-1.93,0.67)  ;
\filldraw [rotate=\i*120]  (-1.93,0.67) circle (0.05) ;
\draw [rotate=\i*120,very thick] (-1.93,0.67) .. controls +(0,-0.3) and +(0,0.2) .. (-1.15,0.47);
\draw [rotate=\i*120,very thick,fleche] (-1.93,-0.47)  -- (1*120+110:3) ; 
}

\def \k {2.3}
\draw [rounded corners=0.5cm, dotted] (0:\k) -- (120:\k) -- (240:\k) -- cycle;

\foreach \i in {1,2,3}
{

\draw  [dashed, very thick,rotate=\i*120]  (-1.93,-0.47) .. controls +(-70:1) and +(-135:0.3) ..   (-1,-1.9) 
.. controls +(45:0.3) and +(-50:1) .. (-1.93,-0.47) ;
\draw  [dashed, very thick,rotate=\i*120]  (-1.93,-0.47)  
.. controls +(-80:1) and +(180:0.3) .. (-1,-2.4) 
.. controls +(0:0.5) and +(180:0.5) .. (0,-0.7) 
.. controls +(0:0.5) and +(180:0.3) .. (1.35,-1.43);
\draw [rotate=\i*120] (-1,-1.3) node {$t_{\i}$};
\draw [rotate=\i*120] (0,-0.5) node {$t'_{\i}$};
}

\end{tikzpicture} 
}\\

The whole family: $T = \{\pm t_{1}, \pm t'_{1}, \pm t_{2}, \pm t'_{2}, \pm t_{3}, \pm t'_{3}\}$

\end{center}
\end{minipage}

%
\bigskip

\bibliographystyle{alpha}
\bibliography{bibliographie}

\begin{thebibliography}{Mat00}

\bibitem[BF93]{barge1993recurrent}
M.~Barge and J.~Franks.
\newblock Recurrent sets for planar homeomorphisms.
\newblock In {\em From Topology to Computation: Proceedings of the Smalefest},
  pages 186--195, 1993.

\bibitem[CB88]{cassonbleiler}
Andrew~J. Casson and Steven~A. Bleiler.
\newblock {\em Automorphisms of surfaces after {N}ielsen and {T}hurston},
  volume~9 of {\em London Mathematical Society Student Texts}.
\newblock Cambridge University Press, Cambridge, 1988.

\bibitem[FH03]{franks2003periodic}
J.~Franks and M.~Handel.
\newblock Periodic points of hamiltonian surface diffeomorphisms.
\newblock {\em Geometry and Topology}, 7(2):713--756, 2003.

\bibitem[FH10]{franks2010periodic}
J.~Franks and M.~Handel.
\newblock Entropy zero area preserving diffeomorphisms of $s^2$.
\newblock {\em http://arxiv.org/abs/1002.0318}, 2010.

\bibitem[Gui94]{guillou1994Brouwer}
L.~Guillou.
\newblock Th{\'e}oreme de translation plane de brouwer et
  g{\'e}n{\'e}ralisations du th{\'e}or{\`e}me de poincar{\'e}-birkhoff.
\newblock {\em Topology}, 33, 1994.

\bibitem[Han99]{handel1999fixed}
M.~Handel.
\newblock A fixed-point theorem for planar homeomorphisms.
\newblock {\em Topology}, 38(2):235--264, 1999.

\bibitem[LC06]{MR2284059}
Patrice Le~Calvez.
\newblock Une nouvelle preuve du th\'eor\`eme de point fixe de {H}andel.
\newblock {\em Geom. Topol.}, 10:2299--2349 (electronic), 2006.

\bibitem[Mat00]{matsumoto2000arnold}
S.~Matsumoto.
\newblock Arnold conjecture for surface homeomorphisms.
\newblock {\em Topology and its Applications}, 104(1-3):191--214, 2000.

\end{thebibliography}

\end{document}